\documentclass[preprint]{elsarticle}
\usepackage{amsmath}
\usepackage{amsthm}
\usepackage{amsfonts}
\usepackage{bm}
\usepackage{mathrsfs}
\usepackage{curves}
\usepackage{color}
\usepackage{subfig}
\usepackage{booktabs}
\usepackage{bookmark}
\usepackage{multirow}
\usepackage{algorithm}
\usepackage{algorithmic}
\usepackage{epstopdf}

\long\def\mOne#1{#1}
\long\def\mTwo#1{#1}
\long\def\mMy#1{#1}

\newtheorem{theorem}{Theorem}[section]
\newtheorem{lemma}[theorem]{Lemma}
\theoremstyle{definition}
\newtheorem{remark}[theorem]{Remark}
\newdefinition{definition}{Definition}
\newdefinition{example}{Example}[section]

\journal{Journal of Computational Physics}
\begin{document}
\begin{frontmatter}
\title{Finite element method for nonlinear Riesz space fractional diffusion equations on 
irregular domains} 
\author[nwpu]{Z. Yang}
\ead{zzyang@mail.nwpu.edu.cn}
\author[nwpu]{Z. Yuan}
\author[nwpu]{Y. Nie\corref{cor1}}
\ead{yfnie@nwpu.edu.cn}
\author[nwpu]{J. Wang}
\author[nwpu]{X. Zhu}
\author[qut]{F. Liu} 

\cortext[cor1]{Corresponding author}
\address[nwpu]{
  Research Center for Computational Science,
  Northwestern Polytechnical University, 
  Xi'an, Shaanxi 710072, China}
\address[qut]{Mathematical Sciences School, Queensland University of Technology, 
QLD.4001, Australia}
\begin{abstract}
In this paper, we consider two-dimensional Riesz space fractional diffusion
equations with nonlinear source term on convex domains. Applying Galerkin
finite element method in space and backward difference method in time, we
present a fully discrete scheme to solve Riesz space fractional diffusion
equations. Our breakthrough is developing an algorithm to form stiffness matrix
on unstructured triangular meshes, which can help us to deal with space
fractional terms on any convex domain. The stability and convergence of the
scheme are also discussed. Numerical examples are given to verify accuracy and
stability of our scheme.
\end{abstract}

\end{frontmatter}

\section[Introduction]{Introduction}\label{sec:introduciton}
In recent years, fractional calculus is becoming more and more popular among
various fields due mainly to its widely applications in science and
engineering, see
\cite{Gorenflo2001,Zaslavsky2002,Podlubny1998,Metzler2000,Kilbas2006}.  In
physics, space fractional derivatives are used to model anomalous diffusion
(super-diffusion and sub-diffusion). In water resources, fractional models are
used to describe chemical and pollute transport in heterogeneous aquifers
\cite{Benson2000}.

Owing to fractional differential equations' various applications,
seeking effective methods to solve them is becoming more and more important. 
There are a
large volume of literatures available on this subject. Researchers have
presented many analytical techniques for solving fractional differential
equations, such as Fourier transform method, Laplace transform method, Mellin
transform method, and Green function method \cite{Kilbas2006}. 
However, it is difficult to find the close
forms of most fractional differential equations, and the close forms are always
represented by special functions, such as Mittag-Leffler function, which means
they are difficult to represent simply and compute directly.  Moreover, most
nonlinear equations are not solvable by analytical methods, so researchers have
to resort to numerical methods.

Over the last few decades, many classical numerical methods have been extended
to solve fractional differential equations, such as finite difference
method~\cite{Liu2013,Zhuang2008,Meerschaert2004,Wang2012}, finite element
method~(FEM)~\cite{Roop2006,Ervin2006,Ervin2007,Jin2015}, and spectral
method~\cite{Chen2014a,Li2009,Li2012,Bueno-Orovio2014}. 

\mTwo{ 
As an efficient method widely used in engineering design and analysis, 
FEM has been deeply studied by a number of scholars to solve fractional 
differential equations.
Ervin and Roop \cite{Ervin2007} defined directional
integrals and directional derivatives, and developed a theoretical framework
for the variational problem of the steady state fractional
advection-dispersion equation on bounded domains in $\mathbb{R}^d$.  Deng
\cite{Deng2009} investigated FEM for the one-dimensional
space and time fractional Fokker-Planck equation. In \cite{Zhang2010}, adopting
FEM, Zhang, Liu and Anh solved one-dimensional symmetric
space-fractional differential equations.  Zhang and Deng \cite{Zhang2012a}
proposed FEM for two-dimensional fractional diffusion equations with time
fractional derivative.  In \cite{Choi2012}, the authors considered FEM for the
space fractional diffusion equation on domains in $\mathbb{R}$.  Deng and
Hesthaven \cite{Deng2013b} proposed a local discontinuous Galerkin method for
the fractional diffusion equation, and offered stability analysis and error
estimates.  Wang and Yang \cite{Wang2013} derived a Petrov-Galerkin weak
formulation to the fractional elliptic differential equation and proved that
the bilinear form is weakly coercive.  Bu et al.\ \cite{Bu2014, Bu2015,
Bu2015a} considered two-dimensional space fractional diffusion equations on
rectangle domains solved by FEM. In \cite{Qiu2015}, Qiu et al. developed
nodal discontinuous Galerkin methods for fractional diffusion equations on 2D
irregular domains and provided stability analysis and error estimates.  Du and
Wang \cite{Du2015} introduced a fast FEM for 2D space-fractional dispersion
equations  by exploiting the structure of stiffness matrix for rectangular mesh
on rectangular domain.
As we can see, many works on FEM are limited in solving
fractional differential equations with linear source term on rectangle domains
with regular meshes. Two-dimensional space fractional problems with nonlinear
source term defined on irregular domains, especially partitioned with
unstructured meshes, are seldom considered, although they are more real and
more useful.}

In this paper, we consider the two-dimensional Riesz space fractional diffusion
equation on convex domain $\Omega$ with initial condition and boundary
condition:
\begin{equation} \label{eq:equation}
  \left\{
  \begin{aligned}
    & \frac{\partial u}{\partial t} = 
      K_x \frac{\partial^{2\alpha}u}{\partial|x|^{2\alpha}} + 
      K_y \frac{\partial^{2\beta}u}{\partial|y|^{2\beta}} + F(u) + f(x, y, t),
      \quad (x,y,t) \in \Omega \times (0, T],\\
    & u(x, y, 0) = \varphi(x, y), \quad (x,y) \in \Omega, \\
  & u(x, y, t) = 0, \quad (x,y,t) \in \partial\Omega \times (0, T],
  \end{aligned}
  \right.
\end{equation}
where $\frac{1}{2} < \alpha,\beta < 1$, 
$K_x > 0, K_y > 0$, and {$F(u)\in C^1(\Theta)$ is a nonlinear function 
($\Theta$ is a proper close domain)}.
Boundaries of $\Omega$ are defined as follows (Fig.~\ref{fig:domain}): 
\begin{figure}
  \centering
  \includegraphics[width=0.3\textwidth]{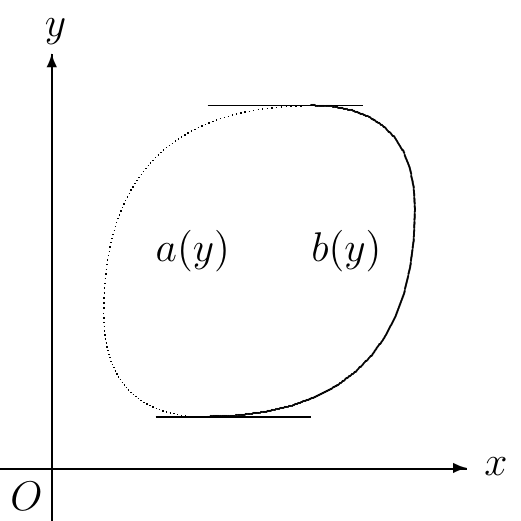}
  \hskip15pt
  \includegraphics[width=0.3\textwidth]{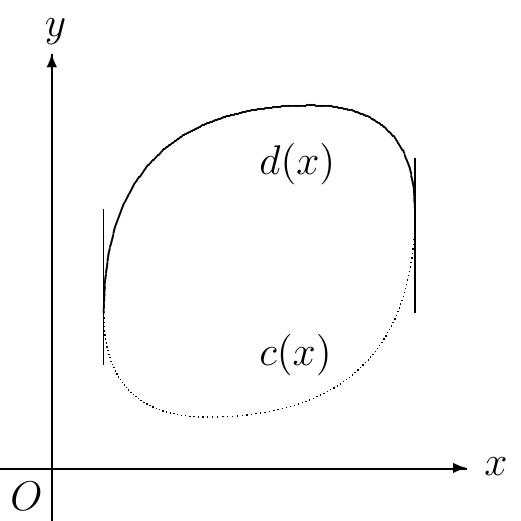}
  \caption{Convex domain $\Omega$ with its boundary $a(y), b(y), c(x), d(x)$.}\label{fig:domain}
\end{figure}
\begin{equation*} \label{eq:boundary}
  \left\{
  \begin{aligned}
    a(y) &= \min \{x: (x, \eta) \in \Omega, \eta=y\}, \\
    b(y) &= \max \{x: (x, \eta) \in \Omega, \eta=y\}, \\
    c(x) &= \min \{y: (\xi, y) \in \Omega, \xi=x\}, \\
    d(x) &= \max \{y: (\xi, y) \in \Omega, \xi=x\}. \\
  \end{aligned}\right.
\end{equation*}
In Eq.~\eqref{eq:equation}, Riesz derivatives \cite{Kilbas2006} $\frac{\partial^{2\alpha}u}{\partial|x|^{2\alpha}}$
and $\frac{\partial^{2\beta}u}{\partial|y|^{2\beta}}$ are defined by
\mTwo{
\begin{equation}
  \begin{aligned}
    \frac{\partial^{2\alpha}u(x,y, t)}{\partial|x|^{2\alpha}} &= 
    - c_{\alpha} \left({}_{a(y)}D_{x}^{2\alpha}u(x,y,t) + {}_{x}D_{b(y)}^{2\alpha}u(x,y,t)\right), \\
    \frac{\partial^{2\beta}u(x,y,t)}{\partial|y|^{2\beta}} &= 
    - c_{\beta} \left({}_{c(x)}D_{y}^{2\beta}u(x,y,t) + {}_{y}D_{d(x)}^{2\beta}u(x,y,t)\right),
  \end{aligned}
\end{equation} }
where $c_{\alpha} = \frac{1}{2\cos(\alpha\pi)}$, $c_{\beta} = \frac{1}{2\cos(\beta\pi)}$,
and the operators ${}_{a(y)}D_{x}^{\mu}u(x,y)$, ${}_{x}D_{b(y)}^{\mu}u(x,y)$, 
${}_{c(x)}D_{y}^{\mu}u(x,y)$, ${}_{y}D_{d(x)}^{\mu}u(x,y)$ ($n-1 < \mu < n, n \in \mathbb{N}$) 
are defined as~\cite{Podlubny1998}
\mTwo{
\begin{equation*}
  \begin{aligned}
    {}_{a(y)}D_{x}^{\mu}u(x,y,t) &= \frac{1}{\Gamma(n-\mu)}\frac{\partial^n}{\partial x^n} \int_{a(y)}^{x}{(x-s)}^{n-\mu-1}u(s,y,t)ds,  \\
    {}_{x}D_{b(y)}^{\mu}u(x,y,t) &= \frac{{(-1)}^n}{\Gamma(n-\mu)}\frac{\partial^n}{\partial x^n} \int_{x}^{b(y)}{(s-x)}^{n-\mu-1}u(s,y,t)ds, \\
    {}_{c(x)}D_{y}^{\mu}u(x,y,t) &= \frac{1}{\Gamma(n-\mu)}\frac{\partial^n}{\partial y^n} \int_{c(x)}^{y}{(y-s)}^{n-\mu-1}u(x,s,t)ds, \\
    {}_{y}D_{d(x)}^{\mu}u(x,y,t) &= \frac{{(-1)}^n}{\Gamma(n-\mu)}\frac{\partial^n}{\partial y^n} \int_{y}^{d(x)}{(s-y)}^{n-\mu-1}u(x,s,t)ds. \\
  \end{aligned}
\end{equation*} }

In this paper, an implicit Galerkin FEM is developed to solve
Eq.~\eqref{eq:equation}, in which the time derivative is discretized by
backward Euler method and the nonlinear term $F(u)$ is approximated by Taylor
formula.  Under suitable conditions, our method is stable and convergent.

The outline of this paper is shown as below. Section~\ref{sec:preliminaries}
gives some notations and lemmas, which will be used later on.  In
Section~\ref{sec:scheme}, we present the backward Euler Galerkin method (BEGM)
and its implementation in detail.  Stability and convergence are investigated in
Section~\ref{sec:stability}.  In Section~\ref{sec:examples}, some numerical
results are tested.  And the last section offers some conclusions on the method and some thoughts
on the future work.

\section[Preliminaries]{Preliminaries}\label{sec:preliminaries}
This section mainly introduce some definitions and lemmas, 
introduced by Ervin and Roop in \cite{Ervin2006,Ervin2007}. We list 
some here for the following sections.  Firstly, we give the
definitions of fractional derivative spaces, 
i.e.\ $J_L^{\mu}(\Omega)$, $J_R^{\mu}(\Omega)$, $J_S^{\mu}(\Omega)$, and
\mTwo{$H^{\mu}(\Omega)$}.
\begin{definition}
  Let $\mu > 0$. Define the seminorm
  \begin{equation*}
    |u|_{J_L^{\mu}(\Omega)} = {\big(
    \|{}_{a(y)}D_{x}^{\mu}u\|^2_{L^2(\Omega)} + 
    \|{}_{c(x)}D_{y}^{\mu}u\|^2_{L^2(\Omega)}  \big)}^{1/2}
  \end{equation*}
  and the norm
  \begin{equation*}
    \|u\|_{J_L^{\mu}(\Omega)} = 
    {\big(\|u\|^2_{L^2(\Omega)} + |u|^2_{J_L^{\mu}(\Omega)}\big)}^{1/2}
  \end{equation*}
  and denote $J_L^{\mu}(\Omega)$ ($J_{L,0}^{\mu}(\Omega)$) as the closure of 
  $C^{\infty}(\Omega)$ ($C^{\infty}_0(\Omega)$) with respect to 
  $\|\cdot\|_{J_L^{\mu}(\Omega)}$.
\end{definition}

\begin{definition}
  Let $\mu > 0$. Define the seminorm
  \begin{equation*}
    |u|_{J_R^{\mu}(\Omega)} = {\big(
    \|{}_{x}D_{b(y)}^{\mu}u\|^2_{L^2(\Omega)} + 
    \|{}_{y}D_{d(x)}^{\mu}u\|^2_{L^2(\Omega)}  \big)}^{1/2}
  \end{equation*}
  and the norm
  \begin{equation*}
    \|u\|_{J_R^{\mu}(\Omega)} = 
    {\big(\|u\|^2_{L^2(\Omega)} + |u|^2_{J_R^{\mu}(\Omega)}\big)}^{1/2}
  \end{equation*}
  and denote $J_R^{\mu}(\Omega)$ ($J_{R,0}^{\mu}(\Omega)$) as the closure of 
  $C^{\infty}(\Omega)$ ($C^{\infty}_0(\Omega)$) with respect to 
  $\|\cdot\|_{J_R^{\mu}(\Omega)}$.
\end{definition}

\begin{definition}
  Let $\mu \neq n - 1/2$, $n\in \mathbb{N}$. Define the seminorm 
  \begin{equation*}
    |u|_{J_S^{\mu}(\Omega)} = {\big(
    |({}_{a(y)}D_{x}^{\mu}u, {}_{x}D_{b(y)}^{\mu}u)| + 
    |({}_{c(x)}D_{y}^{\mu}u, {}_{y}D_{d(x)}^{\mu}u)| 
    \big)}^{1/2}
  \end{equation*}
  and the norm
  \begin{equation*}
    \|u\|_{J_S^{\mu}(\Omega)} = {\big(
      \|u\|^2_{L^2(\Omega)} + |u|^2_{J_S^{\mu}(\Omega)} \big)}^{1/2}
  \end{equation*}
  and denote $J_S^{\mu}(\Omega)$ ($J_{S,0}^{\mu}(\Omega)$) as the closure of 
  $C^{\infty}(\Omega)$ ($C^{\infty}_0(\Omega)$) with respect to 
  $\|\cdot\|_{J_S^{\mu}(\Omega)}$.
\end{definition}

\begin{definition}
  Let $\mu > 0$. Define the seminorm
  \begin{equation*}
    |u|_{H^{\mu}(\Omega)} = \| |\omega|^{\mu}\mathcal{F}(\hat u)(\omega)\|_{L^2(\mathbb{R}^2)}
  \end{equation*}
  and the norm
  \begin{equation*}
    \|u\|_{H^{\mu}(\Omega)} = {\big(
      \|u\|^2_{L^2(\Omega)} + |u|^2_{H^{\mu}(\Omega)} \big)}^{1/2}
  \end{equation*}
  where $\mathcal{F}(\hat u)(\omega)$ is the Fourier transformation of function 
  $\hat u$, $\hat u$ is the zero extension of $u$ outside $\Omega$, 
  and denote $H^{\mu}(\Omega)$ ($H_{0}^{\mu}(\Omega)$) as the closure of 
  $C^{\infty}(\Omega)$ ($C^{\infty}_0(\Omega)$) with respect to 
  $\|\cdot\|_{H^{\mu}(\Omega)}$.
\end{definition}
Based on these definitions, the following lemma shows that the spaces 
$J_{L,0}^{\mu}(\Omega)$, $J_{R,0}^{\mu}(\Omega)$, $J_{S,0}^{\mu}(\Omega)$ 
and $H_0^{\mu}(\Omega)$ are equivalent with equivalent seminorms and norms if 
$\mu \ne n - 1/2$. 
\begin{lemma}[\cite{Ervin2007}]\label{lemma:relation}
  Let $\mu \ne n - 1/2 $ $(n\in\mathbb{N})$, and \mTwo{$u \in
    J_{L,0}^{\mu}(\Omega) \cap J_{R,0}^{\mu}(\Omega)
  \cap H^{\mu}_0{(\Omega)}$}. Then
  there exist positive constants $C_1$ and $C_2$ independent of $u$ such that
  \begin{equation}
    C_1 |u|_{H^{\mu}(\Omega)} \le  
    \max\big\{|u|_{J_L^{\mu}(\Omega)},|u|_{J_R^{\mu}(\Omega)}\big\}
    \le C_2 |u|_{H^{\mu}(\Omega)}. 
  \end{equation}
\end{lemma}
We also have the fractional Poincar\'e-Friedrichs inequalities.
\begin{lemma}\label{lemma:equivalent}
  For $u \in H_0^{\mu}(\Omega)$ and $ 0 < s < \mu$, we have
  \mTwo{
  \begin{equation}
    \begin{aligned}
      \|u\|_{L^2(\Omega)} \le C_1\|{}_{a(y)}D_{x}^{s}u\|_{L^2(\Omega)} 
      \le C_2\|{}_{a(y)}D_{x}^{\mu}u\|_{L^2(\Omega)}, \\
      \|u\|_{L^2(\Omega)} \le C_3\|{}_{c(x)}D_{y}^{s}u\|_{L^2(\Omega)} 
      \le C_4\|{}_{c(x)}D_{y}^{\mu}u\|_{L^2(\Omega)},
    \end{aligned}
  \end{equation} }
  where $C_1$, $C_2$, $C_3$, and $C_4$ are positive constants independent of $u$.
\end{lemma}
\begin{proof}
  See Theorem $3.1.9$ in~\cite{Roop2004}.
\end{proof}
\begin{lemma}[\cite{Ervin2007}]\label{lemma:23}
  Let $\mu > 0$, $u \in J_{L,0}^{\mu}(\Omega) \cap J_{R,0}^{\mu}(\Omega)$. Then
  \mTwo{
  \begin{equation}
    \begin{aligned}
    \big({}_{a(y)}D_{x}^{\mu}u(x,y), {}_{x}D_{b(y)}^{\mu}u(x,y)\big) 
    &= \cos(\mu\pi)\|{}_{-\infty}D_{x}^{\mu}\hat u(x,y)\|^2_{L^2(\mathbb{R}^2)}  \\
    &= \cos(\mu\pi)\|{}_xD_{+\infty}^{\mu}\hat u(x,y)\|^2_{L^2(\mathbb{R}^2)}, \\
    \big({}_{c(x)}D_{y}^{\mu}u(x,y), {}_{y}D_{d(x)}^{\mu}u(x,y)\big) 
    &= \cos(\mu\pi)\|{}_{-\infty}D_{y}^{\mu}\hat u(x,y)\|^2_{L^2(\mathbb{R}^2)}  \\
    &= \cos(\mu\pi)\|{}_yD_{+\infty}^{\mu}\hat u(x,y)\|^2_{L^2(\mathbb{R}^2)}, \\
    \end{aligned}
  \end{equation}}
where $\hat u$ is the extension of $u$ by zero outside $\Omega$.
\end{lemma}
For the proof of this lemma, see \cite{Roop2004} for more details. 

\begin{lemma}[\cite{Zhang2010}]\label{lemma:var}
  Let $1/2 < \mu < 1 $. $u,v \mTwo{\in H_0^{\mu}(\Omega) \cap H_0^{2\mu}(\Omega)}$. Then
  \begin{equation}\label{eq:lvar}
    \begin{aligned}
      \big({}_{a(y)}D_{x}^{2\mu}u(x,y), v(x,y)\big) &= \big({}_{a(y)}D_{x}^{\mu}u(x,y), {}_{x}D_{b(y)}^{\mu}v(x,y)\big), \\
      \big({}_{x}D_{b(y)}^{2\mu}u(x,y), v(x,y)\big) &= \big({}_{x}D_{b(y)}^{\mu}u(x,y), {}_{a(y)}D_{x}^{\mu}v(x,y)\big).
    \end{aligned}
  \end{equation}
\end{lemma}
\mTwo{
\begin{proof}
  Assuming $u, v \in C^\infty_0(\Omega)$, by the property of fractional derivatives, we have
  \cite[see formulas 2.4.12 and 2.4.13 in page 92]{Kilbas2006}
  \begin{equation*}
    {}_{a(y)}D^\alpha_x u = {}_{a(y)}J^{1-\alpha}_x D_xu, \quad 0 < \alpha < 1,
  \end{equation*}
  where $D_x$ represent the classical derivative of variable $x$, and
  ${}_{a(y)}J^{\mu}_x$ is left fractional integral operator defined by
  \begin{equation*}
    {}_{a(y)}J_{x}^{\mu}u(x,y)=\frac{1}{\Gamma(\mu)}\int_{a(y)}^{x}{(x-s)}^{\mu-1}u(s,y)ds,
    \quad \mu > 0.
  \end{equation*}
  Similarly, define right fractional integral operator
  \begin{equation*}
    {}_{x}J_{b(y)}^{\mu}u(x,y)=\frac{1}{\Gamma(\mu)}\int_{x}^{b(y)}{(s-x)}^{\mu-1}u(s,y)ds,
    \quad \mu > 0.
  \end{equation*}
  According to the definitions of the integral operators defined above, we have 
  \cite[see formula 2.1.30 in page 73]{Kilbas2006}
  \begin{equation*}
    \begin{aligned}
    {}_{a(y)}J_{x}^{\alpha + \beta}u(x,y) = {}_{a(y)}J_{x}^{\alpha}{}_{a(y)}J_{x}^{\beta}u(x,y), \\
    {}_{x}J_{b(y)}^{\alpha + \beta}u(x,y) = {}_{x}J_{b(y)}^{\alpha}{}_{x}J_{b(y)}^{\beta}u(x,y).
    \end{aligned}
  \end{equation*}
  Then
  \begin{equation*}
    \begin{aligned}
      \big({}_{a(y)}D_{x}^{2\mu}u(x,y), v(x,y)\big) 
      &= \big(D_x^2\ {}_{a(y)}J_{x}^{2-2\mu}u(x,y), v(x,y)\big) \\
      &= \big(D_x\ {}_{a(y)}J_{x}^{2-2\mu}u(x,y), -D_x\ v(x,y)\big) \\
      &= \big({}_{a(y)}J_{x}^{2-2\mu}D_x u(x,y), -D_x\ v(x,y)\big). \\
    \end{aligned}
  \end{equation*}
  Applying Corollary~2.1 in Ref.~\cite{Ervin2007}, we can deduce that
  \begin{equation*}
    \begin{aligned}
      \big({}_{a(y)}D_{x}^{2\mu}u(x,y), v(x,y)\big) 
      &= \big({}_{a(y)}J_{x}^{1-\mu}D_x u(x,y), -{}_{x}J_{b(y)}^{1-\mu}D_x\ v(x,y)\big) \\
      &= \big(D_x\ {}_{a(y)}J_{x}^{1-\mu}u(x,y), -D_x\ {}_{x}J_{b(y)}^{1-\mu}v(x,y)\big) \\
      &=  \big({}_{a(y)}D_{x}^{\mu}u(x,y), {}_{x}D_{b(y)}^{\mu}v(x,y)\big).
    \end{aligned}
  \end{equation*}

  Dense argument yields the first equivalent relation in Eq.~\eqref{eq:lvar}. 
  The second identity is proved similarly.
\end{proof} }
The formulas in Lemma~\ref{lemma:var}, which are used to construct the
stiffness matrix in finite element method, are similar to the formula of
integration by parts, but for fractional derivatives.
\mTwo{
\begin{lemma}[\cite{Ervin2007}]\label{lemma:add}
  Let $\mu > 0, \mu \ne n-1/2$. Then for all $u\in  J_{L,0}^{\mu}(\Omega) \cap J_{R,0}^{\mu}(\Omega)$, following inequalities hold
\begin{equation}
  \begin{aligned}
  C_1\|{}_{a(y)}D_{x}^{\alpha}u\|^2 \le \|{}_{x}D_{b(y)}^{\alpha}u\|^2 
    \le C_2\|{}_{a(y)}D_{x}^{\alpha}u\|^2,  \\
  C_3\|{}_{c(x)}D_{y}^{\beta}u\|^2 \le \|{}_{y}D_{d(x)}^{\beta}u\|^2 
    \le C_4\|{}_{c(x)}D_{y}^{\beta}u\|^2,  \\
  \end{aligned}
\end{equation}
where $C_1, C_2, C_3, C_4 > 0$ are independent with $u$.
\end{lemma}
\begin{proof}
  See Lemma 5.4 in Ref.~\cite{Ervin2007}. 
\end{proof}}

\section{Discrete scheme and implementation}\label{sec:scheme} In this section,
we present the detail of BEGM and then analyze it briefly.  We begin with the
variational formulation of Eq.~\eqref{eq:equation}: 

{\it Find $u(t) \in U$ such that
\begin{equation}\label{eq:variation}
  \begin{aligned}
    (u_t, v) + a(u, v) &= l(v),     &\forall v \in V, t \in (0, T],  \\
    (u(\cdot, 0), v)            &= (u_0, v), &\forall v \in V,
  \end{aligned}
\end{equation}
where 
$U = L^2(0, T; V)$, $V = H_{0}^{\alpha}(\Omega) \cap H_{0}^{\beta}(\Omega)$,
and $a(u, v)$, $l(v)$ are \mTwo{given} as
\begin{equation} \label{eq:auv}
  \begin{aligned}
    a(u,v) 
    &= K_x c_{\alpha} \big(({}_{a(y)}D_{x}^{\alpha}u, {}_{x}D_{b(y)}^{\alpha}v) 
    + ({}_{x}D_{b(y)}^{\alpha}u, {}_{a(y)}D_{x}^{\alpha}v)\big) \\
    &\quad + K_y c_{\beta}
    \big(({}_{c(x)}D_{y}^{\beta}u, {}_{y}D_{d(x)}^{\beta}v) 
    + ({}_{y}D_{d(x)}^{\beta}u, {}_{c(x)}D_{y}^{\beta}v)\big),  \\
  \end{aligned}
\end{equation}
\begin{equation}
    l(v) = \int_{\Omega}F(u)vdx + \int_{\Omega}fvdx.
  \end{equation} }

According to properties of fractional derivatives, $a(u, v)$ is bilinear,
continuous and coercive, which will be proved in Section~\ref{sec:stability}.
In the following sections, assume that the domain $\Omega$ is polygonal such
that the boundary is exactly represented by boundaries of triangles.  Let
${\{\mathcal{T}_h\}}$ be a family of shape regular triangulations of $\Omega$,
and $h$ be the maximum diameter of elements in $\mathcal{T}_h$.  For finite
element methods, the idea is to approximate Eq.~\eqref{eq:auv} by
conforming, finite dimensional space $V_h \in V$.  Then we define the test
space 
\begin{equation}
  V_h = \{v_h: v_h \in C(\Omega),v_h \in V, v_h|_K\in P_s(K), \forall K \in \mathcal{T}_h\}.
\end{equation}
where $P_s(K)$ is the set of polynomials of which degrees are at most $s$ in $K$.

In the following subsections, let $\tau$ be the time step size and $n_T$ be a
positive integer with $\tau = T/n_T$ and \mOne{$t_n = n\tau$} for $n = 0, 1,
\dots, n_T$.  For the function $u(x, y, t)$ and $f(x, y, t)$, denote $u^n =
u^n(\cdot)=u(\cdot, t_n)$, $f^n = f(\cdot, t_n)$, and $\bar \partial_t u^n =
\frac{u^n-u^{n-1}}{\tau}$.

\subsection{Backward Euler Galerkin method}\label{ssec:begm}
Using backward Euler method on Eq.~\eqref{eq:variation}, 
\mTwo{we get a semi-discrete approximation} for Eq.~\eqref{eq:equation}
\begin{equation}\label{eq:begm}
  (\bar \partial_t u^n, v_h) + a(u^n, v_h) = \big(F(u^n),v_h\big) 
  + \big(f(x,y,n\tau), v_h\big), \quad \forall v_h \in V_h.
\end{equation}
Because of the nonlinear term $F(u)$, solving Eq.~\eqref{eq:begm} is more
difficult than the linear case.  Here, a linearization method is suggested to
approximate $F(u)$ accurately.
\mTwo{
Assuming $F(u) \in C^1(\Theta)$, $F''(u) \in L^\infty(\Theta)$, and $u_t(x,t)$
is bounded, by Taylor's formula we obtain}
\begin{equation} \label{eq:taylor1}
  F(u^{n}) = F(u^{n-1}) + F'(u^{n-1})(u^{n} - u^{n-1}) + O(\tau^2).
\end{equation}
Insert \eqref{eq:taylor1} to \eqref{eq:begm} and drop the term $O(\tau^2)$,
then we have
\begin{equation*} 
  \begin{aligned}
    (\bar \partial_t u^n, v_h) &+ a(u^n, v_h) - \big(F'(u^{n-1})u^{n}, v_h\big) \\
   &= \big(F(u^{n-1}) - F'(u^{n-1})u^{n-1}, v_h\big) + (f^{n},v_h). 
  \end{aligned}
\end{equation*}
\mTwo{So we get the fully-discrete scheme}:
find $u_h^n \in V_h$ for $n = 1,2,\dots, n_T$ such that
\begin{equation}\label{eq:scheme1}
  \left\{
  \begin{aligned}
   &(\bar \partial_t u_h^n, v_h)+ a(u_h^n, v_h) - \big(F'(u_h^{n-1})u_h^{n}, v_h\big)\\
   & \qquad \qquad = \big(F(u_h^{n-1}) - F'(u_h^{n-1})u_h^{n-1}, v_h\big) + (f^{n},v_h), 
     \quad v_h \in V_h,\\
   &u_h^0 = P\varphi(x, y),
  \end{aligned}
  \right.
\end{equation}
where $P$ is a projection operator. 
\mTwo{
We have obtained the backward Euler Galerkin method (BEGM) as desired.}

\subsection{Implementation of BEGM}\label{ssec:implementation}
Here, we turn to the implementation of BEGM, which works well on any convex
domain with unstructured meshes.

For finite element subspace $V_h$, the set of nodes,
$\{(x_k, y_k):k\in\mathcal{N}\}$, is assumed to consist of the vertices of the
principal lattice on each of the elements and includes the vertices of the
elements. Let $\varphi_k(x_l, y_l) = \delta_{kl}$, $k,l\in \mathcal{N}$, where
$\delta_{kl}$ is the Kronecker symbol, be the \mMy{nodal} basis function. For
each time step, we can expand the discrete solution $u_h^n$ as
\begin{equation}\label{eq:uhn}
  u^n_h = \sum\limits_{k\in \mathcal{N}} u^n_k \varphi_k(x, y).
\end{equation}
where $u^n_k$ is unknown coefficients.  Inserting~\eqref{eq:uhn}
into~\eqref{eq:scheme1}, we obtain

\begin{equation}
  \begin{aligned}
\sum\limits_{k\in \mathcal{N}} \Big( (\varphi_k, \varphi_l)  +&  
\tau  a(\varphi_k, \varphi_l)  
- \tau   \left(F'(u_h^{n-1})\varphi_k, \varphi_l\right) \Big) u^n_k \\
&= \sum\limits_{k\in \mathcal{N}} \Big( (\varphi_k, \varphi_l) 
- \tau \left(F'(u^{n-1}_h)\varphi_k, \varphi_l\right) \Big) u^{n-1}_k \\
&\quad\mTwo{+ \tau \left(F(u^{n-1}_h), \varphi_l\right)} + \tau(f^{n}, \varphi_l), \quad \forall l \in \mathcal{N},
  \end{aligned}
\end{equation}
which we can write in the matrix-vector form as
\begin{equation}
  \begin{aligned}
    (\bm{M} + \tau \bm{A} - \tau \bm{D}^{n-1})\bm{u}^n = 
    (\bm{M} - \tau \bm{D}^{n-1})\bm{u}^{n-1} \mTwo{+ \tau \bm{b}_1^{n-1} + \tau \bm{b}_2^{n}},
  \end{aligned}
\end{equation}
where ${\bm u^n}=(u_k^n)$ is the unknown vector, $\bm{M}=(m_{kl})$ is the mass matrix with elements $m_{kl} = (\varphi_l, \varphi_k)$, 
$\bm{A}=(a_{kl})$ the stiffness matrix with elements $a_{kl} = a(\varphi_l, \varphi_k)$, 
$\bm{D}^{n-1}=(d_{kl}^{n-1})$ the matrix with elements 
$d_{kl}^{n-1} = \big(F'(u^{n-1})\varphi_l, \varphi_k\big)$, 
$\bm{b}_1^{n-1}=(b_{1,k})$ the vector with elements 
$b_{1,k}^{n-1} = \big(F(u^{n-1}), \varphi_k\big)$, and
$\bm{b}_2^{n}=(b_{2,k}^{n})$ the vector with elements $b_{2,k}^{n} = (f^{n}, \varphi_k)$.
Among them, the matrix $\bm{A}$ is more difficult to calculate because 
of the non-locality of fractional derivatives. 

Now, we focus on the computation of $\bm{A}$. The elements of $\bm{A}$ is 
\begin{equation}\label{eq:comauv} 
  \begin{aligned}
    a_{kl} = a(\varphi_l,\varphi_k) &= K_x c_{\alpha}
    \big(({}_{a(y)}D_{x}^{\alpha}\varphi_l, {}_{x}D_{b(y)}^{\alpha}\varphi_k) 
    + ({}_{x}D_{b(y)}^{\alpha}\varphi_l, {}_{a(y)}D_{x}^{\alpha}\varphi_k)\big) \\
    &\quad + K_y c_{\beta}
    \big(({}_{c(x)}D_{y}^{\beta}\varphi_l, {}_{y}D_{d(x)}^{\beta}\varphi_k) 
    + ({}_{y}D_{d(x)}^{\beta}\varphi_l, {}_{c(x)}D_{y}^{\beta}\varphi_k)\big).  \\
  \end{aligned}
\end{equation}
Considering the similarity of four terms in the right hand of~\eqref{eq:comauv}, 
we only illustrate the computing process of 
$({}_{a(y)}D_{x}^{\alpha}\varphi_l, {}_{x}D_{b(y)}^{\alpha}\varphi_k)$ as an example.
Using Gaussian quadrature, we obtain
\begin{equation}
  \begin{aligned}
    ({}_{a(y)}D_{x}^{\alpha}\varphi_l, {}_{x}D_{b(y)}^{\alpha}\varphi_k)
    &= \int_{\varOmega}{}_{a(y)}D_{x}^{\alpha}\varphi_l\; {}_{x}D_{b(y)}^{\alpha}\varphi_k dxdy  \\
    &= \sum_{K\in\mathcal{T}}\int_{K}{}_{a(y)}D_{x}^{\alpha}\varphi_l \;{}_{x}D_{b(y)}^{\alpha}\varphi_k dxdy \\
    &\approx \sum_{K\in\mathcal{T}}\sum_{(x_i,y_i)\in G_K} \omega_i 
      \; {}_{a(y)}D_{x}^{\alpha}\varphi_l|_{(x_i, y_i)} 
      \; {}_{x}D_{b(y)}^{\alpha}\varphi_k|_{(x_i, y_i)}
  \end{aligned}
\end{equation}
where $G_K$ is the set of all Gaussian points in element $K$, 
and $\omega_i$ is weight of Gaussian point $(x_i, y_i)$. 
How to compute ${}_{a(y)}D_{x}^{\alpha}\varphi_l|_{(x_i, y_i)}$ and
${}_{x}D_{b(y)}^{\alpha}\varphi_k|_{(x_i, y_i)}$ is the hardest and most
\mMy{critical} issue.  

\mOne{
First of all, we need to
locate the intersection points between the integral path
and the element boundaries. Unstructured meshes make this work
more difficult. To improve the searching efficiency, we should avoid
searching elements aimlessly. 
Algorithm~\ref{algorithm1} shows how to calculate the intersection points 
$(x_i^j, y_i)$ on the integral path for $x$-left
fractional derivative of basis function at Gaussian points in $K$. For other
fractional derivatives, the algorithms are similar.}

\renewcommand{\algorithmicensure}{\textbf{Output:}}
\begin{algorithm}
      \mTwo{
    \caption{Calculate integral path for $x$-left fractional derivative of basis
      function at Gaussian points in ${K}$}\label{algorithm1}
    \begin{algorithmic}[1]
      \REQUIRE Triangulation  $\mathcal{T}$ with its vertices set $V_\mathcal{T}$, Element $K \in \mathcal{T}$,  Gaussian points $G_{K}$ on $K$.
    \ENSURE Ordered intersection points set $I_i$ for each Gaussian point in $G_{K}$.
    \STATE Let the vertices of $K$ be $(x^K_1, y^K_1), (x^K_2, y^K_2), (x^K_3, y^K_3)$.
    \STATE Set $x_{min} = \min\{x^K_1, x^K_2, x^K_3\}$, $x_{max} = \max\{x^K_1, x^K_2, x^K_3\}$.
    \STATE Set $y_{min} = \min\{y^K_1, y^K_2, y^K_3\}$, $y_{max} = \max\{y^K_1, y^K_2, y^K_3\}$.
      \STATE Set $\Omega_{K}$ be the influence domain 
        $\{(x,y) : y\in[y_{min}, y_{max}], x < x_{max}\}$.
      \STATE Set $E_{K} = \{K' \in \mathcal{T} :  K' \cap \Omega_K \ne
    \emptyset\}$
        \FOR{each Gaussian points $(x_i, y_i) \in G_{K}$}
          \FOR{each element $K' \in E_{K}$}
          \STATE Get intersection points  of edges of $K'$ with line $y = y_i, x<x_{max}$. 
          \STATE Update $I_i$ by the intersection points.
          \ENDFOR
          \STATE Sort $I_i$ and erase repeated points in $I_i$.
        \ENDFOR
        \RETURN $\{I_i\}$ 
    \end{algorithmic}}
\end{algorithm}

\begin{figure}
  \centering
  \includegraphics[width=0.28\textwidth]{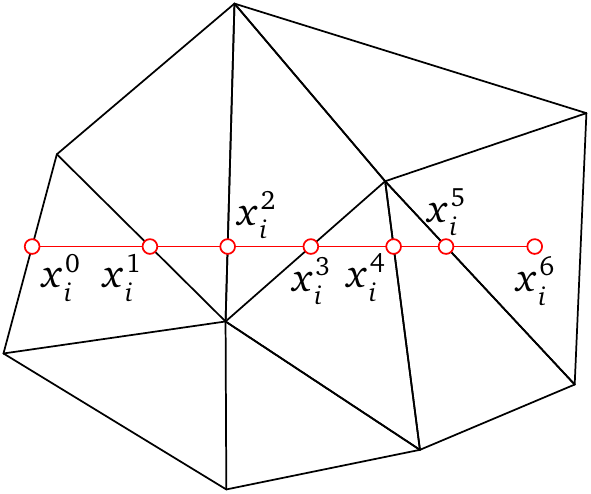}
  \caption{The path to calculate the left derivative. ($x_i^0 = a(y_i)$, $x_i^6 = x_i$) }\label{fig:mesh}
\end{figure}
\mOne{As long as we have found the intersection points, we can compute the value of 
 ${}_{a(y)}D_{x}^{\alpha}\varphi_l|_{(x_i, y_i)}$ and
${}_{x}D_{b(y)}^{\alpha}\varphi_k|_{(x_i, y_i)}$. 
Take ${}_{a(y)}D_{x}^{\alpha}\varphi_l|_{(x_i, y_i)}$ as an example.
Suppose the segment $y=y_i, a(y_i) \le x \le x_i$ intersects
with edges of all triangles at $n_i$ points, and arrange them orderly like
$x_i^0 < x_i^1 <x_i^2< \dots<x_i^{n_i}$, as show in Fig.~\ref{fig:mesh}.}
The value of  ${}_{a(y)}D_{x}^{\alpha}\varphi_l|_{(x_i, y_i)}$ is
\begin{equation}
  \begin{aligned}
    {}_{a(y)}D_{x}^{\alpha}\varphi_l|_{(x_i, y_i)}
    & = {}_{a(y_i)}D_{x}^{\alpha}\varphi_l(x,y_i)|_{x = x_i} \\
    & = {\Big(\frac{1}{\Gamma(1-\alpha)}\frac{d}{dx}
    \int_{a(y_i)}^{x}{(x-t)}^{-\alpha}\varphi_l(t, y_i)\,dt\Big)}_{x=x_i} \\
    & = \sum^{n_i}_{j=1} {\Big(\frac{1}{\Gamma(1-\alpha)}\frac{d}{dx}\int_{x_i^{j-1}}^{x_i^j}
    {(x-t)}^{-\alpha}\varphi_l(t, y_i) \, dt\Big)}_{x=x_i} \\
    & = \mTwo{\sum^{n_i}_{j=1} S_j\big(x;(x_i,y_i)\big)\big|_{x = x_i}}.
    \end{aligned}
\end{equation}
where 
\begin{equation}
   S_j(x;(x_i,y_i)) = \frac{1}{\Gamma(1-\alpha)}\frac{d}{dx}\int_{x_i^{j-1}}^{x_i^j}
    {(x-t)}^{-\alpha}\varphi_l(t, y_i) \, dt.
\end{equation}
\mMy{In the following, we demonstrate  how to calculate $S_j(x;(x_i,y_i))$}.  
Integrating by parts, we notice the fact that
\begin{equation}\label{eq:sa}
  \begin{aligned}
    \frac{d}{dx}\int_{a}^{b}(x-t)^{-\alpha}f(t)\,dt =&\ - f(t)(x-t)^{-\alpha}\big|_{t=a}^{t=b} \\
    &+ \mMy{\frac{1}{1-\alpha}}
    \frac{d}{dx}\int_{a}^{b}(x-t)^{1-\alpha}f'(t)\,dt, \quad x\notin [a,b],
  \end{aligned}
\end{equation}
and
\begin{equation}\label{eq:sb}
    \frac{d}{dx}\int_{a}^{x}(x-t)^{-\alpha}f(t)\,dt = f(a)(x-a)^{-\alpha} 
    + \mMy{\frac{1}{1-\alpha}}\frac{d}{dx}\int_{a}^{x}(x-t)^{1-\alpha}f'(t)\,dt.
\end{equation}
when  $f(x) \in C^1[a,b]$ and $\alpha < 1$.
\mTwo{It is obvious that basis function $\varphi_l(x, y_i)$ is infinitely differentiable in
  $[x_i^{j-1},x_i^j]$ for fixed $y_i$. Define $f(x)=\varphi_l(x, y_i)$, then
  $f(x) \in C^{\infty}[x_i^{j-1},x_i^j]$.}  Using formulas~\eqref{eq:sa} and
  \eqref{eq:sb}, we have
\begin{equation} \label{eq:pfd}
  \begin{aligned}
    S_j(x;(x_i,y_i)) &= \frac{1}{\Gamma(1-\alpha)}\frac{d}{dx}\int_{x_i^{j-1}}^{x_i^j}(x-t)^{-\alpha}f(t)\,dt \\
    &= \frac{-1}{\Gamma(1-\alpha)}f(t)(x-t)^{-\alpha}|_{t=x_i^{j-1}}^{t=x_i^j} 
      + \frac{-1}{\Gamma(2-\alpha)}f'(t)(x-t)^{1-\alpha}|_{t=x_i^{j-1}}^{t=x_i^j} \\
    &+ \frac{1}{\Gamma(3-\alpha)}\frac{d}{dx}\int_{x_i^{j-1}}^{x_i^j}f''(t)(x-t)^{2-\alpha}\,dt
       \qquad (x_i^j \ne x, x\notin[x_i^{j-1},x_i^j]),\\
  \end{aligned}
\end{equation}
and
\begin{equation} \label{eq:pfd2}
  \begin{aligned}
    S_j(x;(x_i,y_i)) &= \frac{1}{\Gamma(1-\alpha)}\frac{d}{dx}\int_{x_i^{j-1}}^{x}(x - t)^{\alpha}f(t)\,dt \\
    &= \frac{1}{\Gamma(1-\alpha)}f(x_i^{j-1})(x-x_i^{j-1})^{-\alpha} 
      + \frac{1}{\Gamma(2-\alpha)}f'(x_i^{j-1})(x-x_i^{j-1})^{1-\alpha} \\
    &+ \frac{1}{\Gamma(3-\alpha)}\frac{d}{dx}\int_{x_i^{j-1}}^{x}f''(t)(x-t)^{2-\alpha}\,dt,
       \qquad (x_i^j = x).\\
  \end{aligned}
\end{equation}
If we use linear triangular element, i.e.\,$f(x)$ is linear function, then
there are two terms in~\eqref{eq:pfd} without last line. So we can get
the derivative by adding the value in all intervals.

\section{Stability and convergence}\label{sec:stability}
In this section we \mOne{analyze} the stability and convergence of BEGM.
In the following part, let $\lambda = \max \{\alpha, \beta\}$ and 
the constant $C$ may have different value in different context.  

According to the bilinear form 
$a(u, v)$, we define seminorm $|\cdot|_{(\alpha, \beta)}$ and norm 
$\|\cdot\|_{(\alpha, \beta)}$ as follows\cite{Zeng2014}:
\begin{equation}
  \begin{aligned}
    |u|_{(\alpha, \beta)} &= 
    {(K_x \|{}_{a(y)}D_{x}^{\alpha}u\|^2 + K_y \|{}_{c(x)}D_{y}^{\beta}u\|^2)}^{1/2}, \\
    \|u\|_{(\alpha, \beta)} &=  {(\|u\|^2 + |u|_{(\alpha, \beta)}^2)}^{1/2}.
  \end{aligned}
\end{equation}
where \mTwo{$\|\cdot\|=\|\cdot\|_{L_2(\Omega)} = \big(\int_\Omega |\cdot|^2 dx\big)^{1/2}$}. 
Due to Lemma~\ref{lemma:equivalent}, seminorm
$|\cdot|_{(\alpha,\beta)}$ and norm $\|\cdot\|_{(\alpha,\beta)}$ are equivalent
if \mTwo{$u\in H_0^{\alpha}(\Omega) \cap H_0^{\beta}(\Omega)$}, which is
described as the following lemma.
\begin{lemma}\label{lemma:eqnorm}
  Suppose that $\Omega$ is convex domain, $u \in V$. Then there exists positive
  constants $C_1 < 1$ and $C_2$ independent of $u$, such that
  \begin{equation}
    C_1 \|u\|_{(\alpha,\beta)} \le |u|_{(\alpha,\beta)} \le \|u\|_{(\alpha,\beta)} 
        \le C_2 |u|_{H^{\lambda}(\Omega)}. 
  \end{equation}
\end{lemma}
\begin{proof}
  From Lemma~\ref{lemma:equivalent}, we immediately have
  \begin{equation*}
    \|u\| \le C|u|_{(\alpha,\beta)},
  \end{equation*}
  where $C$ is independent of $u$. 
  Therefore, there exists positive constants $C_1 < 1$ independent of $u$, such that
  \begin{equation} \label{eq:localeq}
    C_1 \|u\|_{(\alpha,\beta)} \le |u|_{(\alpha,\beta)}, \quad C_1 = \frac{1}{\sqrt{C^2+1}}
  \end{equation}

  The inequality $|u|_{(\alpha,\beta)} \le \|u\|_{(\alpha,\beta)}$ is obvious
  from their definitions.
  Using Lemma~\ref{lemma:equivalent} again, we find
  \begin{equation}\label{eq:above}
    |u|_{(\alpha,\beta)} \le C|u|_{J_L^{\lambda}(\Omega)}.
  \end{equation}
  Combining inequalities \eqref{eq:localeq}, \eqref{eq:above} and Lemma~\ref{lemma:relation},
  we have
  \begin{equation}
    \|u\|_{(\alpha,\beta)} \le \frac{1}{C_1}|u|_{(\alpha,\beta)} 
    \le \frac{C}{C_1}|u|_{J_L^{\lambda}(\Omega)} \le C_2 |u|_{H^{\lambda}(\Omega)}.
  \end{equation}
  The proof is completed.
\end{proof}

By Lemma~\ref{lemma:eqnorm}, we can obtain the following properties of $a(u,v)$.
\mTwo{
\begin{theorem}
  The bilinear form $a(u, v)$ is continuous and coercive, i.e.\ 
\begin{equation}\label{eq:property}
  a(u, v) \le \mathcal{A} \|u\|_{(\alpha, \beta)}\|v\|_{(\alpha, \beta)}, \medspace
  a(u, u) \ge \mathcal{B} \|u\|_{(\alpha, \beta)}^{2}, \quad 
  \forall u\in H_0^{\alpha}(\Omega) \cap H_0^{\beta}(\Omega),
\end{equation}
where $\mathcal{A}$ and $\mathcal{B}$ are positive constants independent of $u$. 
\end{theorem}
\begin{proof}
According to the definition of $a(u, v)$ and Cauchy-Schwarz inequality, we have
\begin{equation}
  \begin{aligned}
    a(u,v) 
  &\le K_x|c_\alpha| \|{}_{a(y)}D_{x}^{\alpha}u\|\|{}_{x}D_{b(y)}^{\alpha}v\| 
  + K_y|c_{\beta}| \|{}_{c(x)}D_{y}^{\beta}u\|\|{}_{y}D_{d(x)}^{\beta}v\| \\
  &\quad+ K_x|c_\alpha| \|{}_{x}D_{b(y)}^{\alpha}u\| \|{}_{a(y)}D_{x}^{\alpha}v\| 
    + K_y|c_{\beta}| \|{}_{y}D_{d(x)}^{\beta}u\|\| {}_{c(x)}D_{y}^{\beta}v\| \\
  \end{aligned}
\end{equation}
Then by inequality
\begin{equation*}
  a_1 b_1 + a_2b_2 \le \sqrt{a_1^2 + a_2^2}\sqrt{b_1^2+b_2^2}, \quad a_1, b_1, a_2, b_2 > 0,
\end{equation*}
we can obtain
\begin{equation}
  \begin{aligned}
    a(u,v) 
      & \le C|u|_{(\alpha,\beta)}
      \sqrt{K_x\|{}_{x}D_{b(y)}^{\alpha}v\|^2 + K_y\|{}_{y}D_{d(x)}^{\beta}v\|^2 }\\
   &\quad  + C|v|_{(\alpha,\beta)}
   \sqrt{K_x\|{}_{x}D_{b(y)}^{\alpha}u\|^2 + K_y\|{}_{y}D_{d(x)}^{\beta}u\|^2 } \\
  \end{aligned}
\end{equation}
where $C = \max\{|c_\alpha|,|c_\beta|\}$.
By Lemma~\ref{lemma:add}, we know
\begin{equation}
  \begin{aligned}
  \|{}_{x}D_{b(y)}^{\alpha}v\|^2 \le C_1\|{}_{a(y)}D_{x}^{\alpha}v\|^2  \\
  \|{}_{y}D_{d(x)}^{\beta}v\|^2 \le C_2\|{}_{c(x)}D_{y}^{\beta}v\|^2  \\
  \end{aligned}
\end{equation}
where $C_1, C_2$ are  non-negative constants. Then, by Lemma~\ref{lemma:eqnorm}, we have 
\begin{equation}
  a(u,v) \le \mathcal{A}|u|_{(\alpha,\beta)}|v|_{(\alpha,\beta)} 
  \le \mathcal{A}\|u\|_{(\alpha,\beta)}\|v\|_{(\alpha,\beta)},
\end{equation}
where $\mathcal{A} = 2\max\{|c_\alpha|, |c_\beta|\}\sqrt{\max\{C_1,C_2\}}$.

Let's move on to the second inequality. From Lemma~\ref{lemma:23} and Lemma~\ref{lemma:eqnorm},
we have 
\begin{equation}
  \begin{aligned}
  a(u,u) &= 
    2K_x c_{\alpha} ({}_{a(y)}D_{x}^{\alpha}u, {}_{x}D_{b(y)}^{\alpha}u) 
    + 2K_y c_{\beta}({}_{c(x)}D_{y}^{\beta}u, {}_{y}D_{d(x)}^{\beta}u)  \\
    &= K_x \|{}_{-\infty}D_{x}^{\mu}\hat u\|^2_{L^2(\mathbb{R}^2)} +
    K_y\|{}_{-\infty}D_{y}^{\mu}\hat u\|^2_{L^2(\mathbb{R}^2)}    \\
    &\ge {K_x \|{}_{a(y)}D_{x}^{\alpha}u\|^2 + K_y \|{}_{c(x)}D_{y}^{\beta}u\|^2}, \\
    &= |u|_{(\alpha,\beta)}^2 \ge \mathcal{B} \|u\|_{(\alpha,\beta)}^2.
  \end{aligned}
\end{equation}
Now we have proved the properties \eqref{eq:property} of $a(u,v)$.
\end{proof}}
Assume that $F(u) \in C^1(\Theta)$  
with $M_1 = \max\limits_{u\in \Theta}|F(u)|$, and
$M_2 = \max\limits_{u\in \Theta}|F'(u)|$. 
In Eq.~\eqref{eq:scheme1}, let $v_h = u_h^n$, then
\begin{equation}
  \begin{aligned}
    (u_h^n, u_h^n) &+ \tau a(u_h^n, u_h^n) - \tau(F'(u_h^{n-1})u_h^{n}, u_h^n) \\
    &= (u_h^{n-1}, u_h^n) + \tau(F(u_h^{n-1}) - F'(u_h^{n-1})u_h^{n-1}, u_h^n) + \tau(f^{n},u_h^n). 
  \end{aligned}
\end{equation}
Using the property \mTwo{$ a(u_h^n, u_h^n) \ge \mathcal{B} \|u_h^n\|^2_{(\alpha,\beta)} \ge 
\mathcal{B} \|u_h^n\|^2$} and Cauchy-Schwarz inequality yields
\begin{equation}
  \begin{aligned}
    \|u_h^n\|^2 &+ \tau (\mathcal{B} - M_2)\|u_h^n\|^2 \le \|u_h^n\|\|u_h^{n-1}\| \\
    &+ \tau M_1 S \|u_h^n\| + \tau M_2 \|u_h^n\|\|u_h^{n-1}\| + \tau \|f^n\| \|u_h^n\|,
  \end{aligned}
\end{equation}
where $S$ is the positive square root of the area of domain $\Omega$.
Then
\begin{equation} \label{eq:srel}
  \begin{aligned} 
    \|u_h^n\| + \tau (\mathcal{B}- M_2)\|u_h^n\| \le \|u_h^{n-1}\| 
    + \tau M_2 \|u_h^{n-1}\| + \tau \|f^n\| + \tau M_1 S.
  \end{aligned}
\end{equation}
Summing $n$ from $1$ to $k$ in~\eqref{eq:srel} 
\begin{equation}
  \begin{aligned}
    \|u_h^k\| &+ \tau (\mathcal{B} - M_2)\|u_h^k\| \le \|u_h^{0}\| + \tau M_2 \|u_h^{0}\| \\
    &+ \tau (2M_2-\mathcal{B}) \sum\limits_{n=1}^{k-1}\|u_h^{n}\|+ 
    \tau \sum\limits_{n=1}^{k}\|f^n\| + k\tau  M_1 S.
  \end{aligned}
\end{equation}
In case $\mathcal{B} \ge M_2$, we see that
\begin{equation} \label{eq:sdeduce}
  \begin{aligned} 
    \|u_h^k\| & \le C\|u_h^{0}\|+ \tau C \sum\limits_{n=1}^{k-1}\|u_h^{n}\|+ 
    \tau \sum\limits_{n=1}^{k}\|f^n\| + k\tau M_1 S.
  \end{aligned}
\end{equation}
where $C$ is a non-negative constant independent with $u_h$.  In case $\mathcal{B} < 
M_2$, let $\tau < \frac{1}{2(M_2-\mathcal{B})}$, then inequality~\eqref{eq:sdeduce} also holds. 

To conduct analysis on stability, we make use of the following Gr\"onwall inequality.
\begin{lemma}[\cite{Quarteroni2008}]\label{lemma:G}
  Assume that $k_n$ is a nonnegative sequence, $g_0 > 0$, and the non-negative
  sequence $\{ \varphi_n \}$ satisfies $\varphi_0 \le g_0$ and  
  \begin{equation}
    \varphi_n \le g_0 + \sum\limits_{j=0}^{n-1}k_j \varphi_j, \quad n \ge 1.
  \end{equation}
  Then
  \begin{equation}
    \varphi_n \le g_0\exp\Bigg(\sum\limits_{j=0}^{n-1}k_j\Bigg), \quad n \ge 1.
  \end{equation}
\end{lemma}

By Lemma~\ref{lemma:G} and inequality \eqref{eq:sdeduce}, we have 
\begin{equation}
  \|u_h^{k}\| \le q_k \exp\left( \tau C(k-1) \right) \le q_{n_T} \exp (CT),
\end{equation}
where 
\begin{equation}\label{eq:gk1}
q_k = C \|u_h^{0}\|+\tau \sum\limits_{n=1}^{k}\|f^n\| + k\tau M_1 S.
\end{equation}
So the analysis of stability is completed. 
In conclusion, we have proved the following theorem.
\begin{theorem}[stability]
  Suppose that $u_h^n$ are solutions of \eqref{eq:scheme1}, and $F(x)$ satisfies
  {$F(x) \in C^1(\Theta)$, $|F(x)| \le M_1$, and $|F'(x)| \le M_2$}.
  Assume $\tau < \frac{1}{2(M_2 - \mathcal{B})}$ if $\mathcal{B} < M_2$. 
  Then
  \begin{equation}
    \|u_h^{k}\| \le q_k \exp\left( \tau C(k-1) \right) \le q_{n_T} \exp (CT)
  \end{equation}
  where $q_k$ is defined by \eqref{eq:gk1}. 
\end{theorem}

Next, consider the convergence of BEGM. 
First, define projection operator $P_h:V \rightarrow V_h$ with
following property:
\begin{equation} \label{eq:projection}
  a(u - P_h u, v_h) = 0, u \in V, \forall v_h \in V_h.
\end{equation}
In order to exploit the property of the projection operator $P_h$, let us
suppose that there is an interpolation 
$I_h:H^{s+1}(\Omega)\rightarrow V_h$ 
satisfied that
\begin{equation}
  \|u-I_h u\|_{H^{\gamma}(\Omega)} \le Ch^{\mu-\gamma} \|u\|_{H^{\mu}(\Omega)},
  \quad \forall u \in H^{\mu}(\Omega), 
  \quad 0 \le \gamma < \mu \le s+ 1.
\end{equation} 
Then we can deduce an approximation property of $P_h$.
\begin{lemma}\label{lemma:ph}
  If $u \in H^{\mu}(\Omega) \cap V$, $\lambda< \mu \le s + 1$, then the
  following estimate holds:
  \begin{equation}
    |u-P_h u|_{(\alpha,\beta)} \le C h^{\mu - \lambda} \|u\|_{H^{\mu}(\Omega)},
  \end{equation}
  where $C$ is independent of $h$ and $u$.
\end{lemma}
The proof is similar with Lemma~4.4 in \cite{Zeng2014}, so we omit here. 

Let $\theta^n = u_h^{n} - P_h u^n$, $\rho^n = P_h u^n - u^n$, and
$e^n = u_h^n - u^n= \theta^n + \rho^n$. The exact solution $u^n$ satisfies
\begin{equation} \label{eq:a}
  (\partial_t u^n, v_h) + a(u^n, v_h) = \big(F(u^n), v_h\big) + (f^n, v_h), 
  \forall v_h \in V_h, t \in (0, T].
\end{equation}
And $u_h^n$ is the numerical solution of 
\begin{equation} \label{eq:b}
  (\bar \partial_t u_h^n, v_h) + a(u_h^n, v_h) =\big(F(u_h^{n-1}),v_h\big) 
  + \big(F'(u^{n-1}_h)(u_h^n-u_h^{n-1}),v_h\big) + (f^n, v_h), 
  \forall v_h \in V_h.
\end{equation}
Subtracting~\eqref{eq:a} from~\eqref{eq:b}, we obtain
\begin{equation} \label{eq:error}
  \begin{aligned}
    (\bar \partial_t e^n, v_h) &+ a(e^n, v_h) = (u_t^n - \bar\partial_t u^n, v_h) \\
    &+ \big(F(u_h^{n-1}) - F(u^n),v_h\big) + \big(F'(u^{n-1}_h)(u_h^n-u_h^{n-1}),v_h\big). 
  \end{aligned}
\end{equation}
Firstly, we estimate $\|\theta^n\|_{(\alpha,\beta)}$. 
Since $a(\rho^n, v_h) = 0$, the Eq.~\eqref{eq:error} can be written as
\begin{equation} \label{eq:error2}
  \begin{aligned}
    (\bar \partial_t \theta^n, v_h) &+ a(\theta^n, v_h) = (u_t^n - \bar\partial_t u^n, v_h)
    -(\bar\partial_t\rho^n, v_h) \\
    &+\big(F(u_h^{n-1}) - F(u^n),v_h\big) + \big(F'(u^{n-1}_h)(u_h^n-u_h^{n-1}),v_h\big). 
  \end{aligned}
\end{equation}

Let $v_h = \bar \partial_t \theta^n$ in~\eqref{eq:error2}, then
\begin{equation} \label{eq:error3}
  \begin{aligned}
    (\bar \partial_t \theta^n, \bar\partial_t\theta^n) &+ a(\theta^n, \bar\partial_t\theta^n) 
    = (u_t^n - \bar\partial_t u^n, \bar\partial_t\theta^n) 
    -(\bar\partial_t\rho^n, \bar\partial_t\theta^n) \\
    &+\big(F(u_h^{n-1}) - F(u^n),\bar\partial_t\theta^n\big)
    + \big(F'(u^{n-1}_h)(u_h^n-u_h^{n-1}),\bar\partial_t\theta^n\big). 
  \end{aligned}
\end{equation}
Separate $\theta^n$ from the right hand of~\eqref{eq:error3}
by Cauchy-Schwarz inequality
\begin{equation} \label{eq:begin}
  (u_t^n - \bar\partial_t u^n, \bar\partial_t\theta^n) \le
  \frac{1}{2\varepsilon} \|u_t^n - \bar\partial_t u^n\|^2 
  + \frac{\varepsilon}{2} \|\bar\partial_t\theta^n\|^2,
\end{equation}
\begin{equation}
  (\bar\partial_t\rho^n, \bar\partial_t\theta^n) \le
  \frac{1}{2\varepsilon} \|\bar\partial_t\rho^n\|^2 
  + \frac{\varepsilon}{2} \|\bar\partial_t\theta^n\|^2,
\end{equation}
\begin{equation}
  \begin{aligned}
    (F(u_h^{n-1}) - F(u^n),\bar\partial_t\theta^n) 
    &\le \|F(u_h^{n-1}) - F(u^n)\| \|\bar\partial_t\theta^n\| \\
    &\le \|M_2 (u_h^{n-1} - u^n)\| \|\bar\partial_t\theta^n\| \\
    &\le \frac{M_2^2}{2\varepsilon}\|u_h^{n-1}- u^n\|^2
    + \frac{\varepsilon}{2}\|\bar\partial_t\theta^n\|^2, 
  \end{aligned}
\end{equation}
\begin{equation}\label{eq:end}
  (F'(u^{n-1}_h)(u_h^n-u_h^{n-1}),\bar\partial_t\theta^n)
  \le \frac{M_2^2}{2\varepsilon}\|u_h^{n-1}- u_h^n\|^2
  + \frac{\varepsilon}{2}\|\bar\partial_t\theta^n\|^2.
\end{equation}
Let $\varepsilon = 1/2$ in~\eqref{eq:begin}-\eqref{eq:end}, then we see
\begin{equation} \label{eq:estimate1}
  \begin{aligned}
    a(\theta^n, \theta^n) - a(\theta^{n-1}, \theta^{n-1})
    &\le 2\tau \|u_t^n - \bar\partial_t u^n\|^2 + 2\tau\|\bar\partial_t\rho^n\|^2 \\
    &+ 2\tau M_2^2\|u_h^{n-1}- u^n\|^2  + 2\tau M_2^2\|u_h^{n-1}- u_h^n\|^2,
  \end{aligned}
\end{equation}
where we have used the following equality,
\begin{equation} \label{eq:equal}
  a(\theta^n, \bar\partial_t\theta^n) = \frac{1}{2\tau}
  \left(
  a(\theta^n, \theta^n) + a(\theta^n - \theta^{n-1}, \theta^n - \theta^{n-1})
  - a(\theta^{n-1}, \theta^{n-1})\right).
\end{equation}
Also note that the last two terms in~\eqref{eq:estimate1} can be estimated as
\begin{equation} \label{eq:430}
  \left\{
  \begin{aligned}
  &\|u_h^{n-1}- u^n\|^2 \le 3(\|\rho^{n-1}\|^2 + \|\theta^{n-1}\|^2 
  + \|u^{n-1} - u^n\|^2),  \\
  &\|u_h^{n-1}- u_h^n\|^2 \le 5(\|\rho^{n-1}\|^2 + \|\theta^{n-1}\|^2 
  + \|\rho^{n}\|^2 + \|\theta^{n}\|^2 + \|u^{n-1} - u^n\|^2),
  \end{aligned}
  \right.
\end{equation}
here we have used the Minkowski inequality and  
\mTwo{\begin{equation}
  (a_1 + a_2 + \cdots + a_m)^2 \le m(a_1^2+a_2^2+\cdots+a_m^2),
\end{equation}}
where $a_i (i=1,\cdots, m)$ is non-negative real number.
\mMy{The inequalities~\eqref{eq:estimate1}},~\eqref{eq:430} and the fact 
$a(\theta^k, \theta^k) \ge \mathcal{B} \|\theta^{k}\|_{(\alpha,\beta)}^2$
yield
\begin{equation}
  \begin{aligned}
    \mathcal{B}\|\theta^k\|_{(\alpha,\beta)}^2 &\le a(\theta^{0}, \theta^{0})
    + 2\tau \sum\limits_{n=1}^{k}(\|u_t^n - \bar\partial_t u^n\|^2 + \|\bar\partial\rho^n\|^2) \\
    & + 2\tau M_2^2 \Big( 5\|\theta^{k}\|^2 + 13\sum\limits_{n=0}^{k} \|\rho^{n}\|^2 
    + 13\sum\limits_{n=0}^{k-1}\|\theta^{n}\|^2
    + 13\sum\limits_{n=1}^{k} \|u^{n-1} - u^n\|^2  \Big).
  \end{aligned}
\end{equation}
When $\tau \le \frac{\mathcal{B}}{20M_2^2}$, we have
\begin{equation} \label{eq:deduceC}
    \|\theta^{k}\|_{(\alpha,\beta)}^2 \le  C g_k + 
    \tau C'M_2^2\sum\limits_{n=0}^{k-1}\|\theta^{n}\|_{(\alpha,\beta)}^2
\end{equation}
where  \mTwo{$C = \frac{1}{\mathcal{B}-10\tau M_2^2}$,
$C'= \frac{26}{\mathcal{B}-10\tau M_2^2}$}, and $g_k$ is defined as 
\begin{equation}
  \begin{aligned}
    g_k =& a(\theta^0, \theta^0)
    + 2\tau \sum\limits_{n=1}^{k}(\|u_t^n - \bar\partial_t u^n\|^2 + \|\bar\partial\rho^n\|^2) \\
    &+ 26\tau M_2^2 \Big(\sum\limits_{n=0}^{k} \|\rho^{n}\|^2 
    + \sum\limits_{n=1}^{k} \|u^{n-1} - u^n\|^2 \Big).
  \end{aligned}
\end{equation}
By Lemma~\ref{lemma:G} and inequality~\eqref{eq:deduceC}, the following estimation holds
\begin{equation}
  \|\theta^{k}\|_{(\alpha,\beta)}^2 \le  C g_k\exp{(C'TM_2^2)}.
\end{equation}
Now, we need to estimate $g_k$. In order to make it clear, 
we denote $\|\cdot\|_{\mu} = \|\cdot\|_{H^{\mu}(\Omega)}$
and $\|\cdot\|_{0,\mu} = ( \int_0^T \|\cdot\|^2_{\mu}dt )^{1/2}.$
The continuity of $a(u,v)$ and Lemma~\ref{lemma:ph}
imply that 
\begin{equation} \label{eq:begin2}
  \begin{aligned}
    a(\theta^0, \theta^0) &\le C \|\theta^0\|_{(\alpha,\beta)}^2
    \le C \left(\|u^0_h - u(t_0)\|^2_{\lambda}+\|u(t_0)-\mTwo{P_h u(t_0)}\|^2_{(\alpha,\beta)}
    \right) \\
    &\le C \left( \|u^0_h - u(t_0)\|^2_{\lambda} + h^{2\mu-2\lambda}\|u(t_0)\|^2_{\mu} \right).
  \end{aligned}
\end{equation}
For the term with $\|u_t^n - \bar\partial_t u^n\|$, we use Taylor formula and
Cauchy-Schwarz inequality:
\begin{equation}
  \begin{aligned}
    u_t^n - \bar\partial_t u^n 
    &= \frac{1}{\tau}\int_{t_{n-1}}^{t_n}(t-t_{n-1})u_{tt}dt \\
    &\le \frac{1}{\tau}{\bigg(\int_{t_{n-1}}^{t_n}{(t-t_{n-1})}^2dt 
    \int_{t_{n-1}}^{t_n}u_{tt}^2dt\bigg)}^{\frac{1}{2}} \\
    &= {\left(\frac{\tau}{3}\right)}^{\frac{1}{2}} 
    {\bigg(\int_{t_{n-1}}^{t_n}u_{tt}^2dt\bigg)}^{\frac{1}{2}}.
  \end{aligned}
\end{equation}
Then we have
\begin{equation}
  2\tau \sum\limits_{n=1}^{k}\|u_t^n - \bar\partial_t u^n\|^2
  \le 2\tau\sum\limits_{n=1}^{k}
  \int_{\Omega}\bigg(\frac{\tau}{3}\int_{t_{n-1}}^{t_n}u_{tt}^2dt\bigg)dxdy
  = \frac{2\tau^2}{3}\|u_{tt}\|^2_{0,0}.
\end{equation}
Apply Cauchy-Schwarz inequality to the third term of $g_k$, then we obtain
\begin{equation}
  \begin{aligned}
    2\tau \sum\limits_{n=1}^{k}\|\bar\partial_t\rho^n\|^2
    &= 2\tau \sum\limits_{n=1}^{k}\|\frac{1}{\tau}\int_{t_{n-1}}^{t_n}\rho_t dt\|^2 \\
    &\le 2\tau \sum\limits_{n=1}^{k}\int_{\Omega} \bigg(\frac{1}{\tau^2}
    \int_{t_{n-1}}^{t_n}1^2dt \int_{t_{n-1}}^{t_n}\rho_t^2dt\bigg)dxdy \\
    &\le Ch^{2\mu-2\lambda}\|u_t\|_{0,\mu}^2.
  \end{aligned}
\end{equation}
By the definition of $\rho^n$ and the property of projection $P_h$, we have
\begin{equation}
  \begin{aligned}
    6\tau M_2^2 \sum\limits_{n=0}^{k} \|\rho^{n}\|^2 
    &\le 6\tau M_2^2 \sum\limits_{n=0}^{k} \|\rho^{n}\|_{(\alpha,\beta)}^2 \\
    &\le 6\tau M_2^2\sum\limits_{n=0}^{k} Ch^{2\mu-2\lambda}\|u^n\|_{\mu}^2  \\
    &\le CT M_2^2 h^{2\mu-2\lambda}\max\limits_{t_0\le t \le T} \|u(t)\|_{\mu}^2.
  \end{aligned}
\end{equation}
For the last term of $g_k$, we use Cauchy-Schwarz inequality again:
\begin{equation} \label{eq:end2}
  \begin{aligned}
    2\tau M_2^2 \sum\limits_{n=1}^{k} \|u^{n-1} - u^n\|^2 
    &= 2\tau M_2^2 \sum\limits_{n=1}^{k}\int_{\Omega}  
    {\bigg(\int_{t_{n-1}}^{t_n}u_t dt \bigg)}^2 dxdy \\
    &\le 2\tau^2 M_2^2 \int_{\Omega}\left(\int_{t_{0}}^{t_{n_T}}u_t^2 dt\right) dxdy \\
    &= 2\tau^2 M_2^2 \|u_t\|^2_{0,0}. 
  \end{aligned}
\end{equation}
Hence, $g_k$ can be estimated as
\begin{equation} \label{eq:gk}
  \begin{aligned}
    g_k \le& C \|u^0_h - u(t_0)\|^2_{\lambda}  + 
    C \tau^2 \big( \|u_{tt}\|^2_{0,0} + M_2^2 \|u_t\|^2_{0,0} \big) \\
    &+ C h^{2\mu-2\lambda} \big( \|u(t_0)\|^2_{\mu} + \|u_t\|_{0,\mu}^2 
    + T M_2^2 \max_{0 \le t \le T} \|u(t)\|_{\mu}^2  \big).
  \end{aligned}
\end{equation}
Since $u_h^k - u^k = \theta^k + \rho^k$, we have the result
\begin{equation}
  \begin{aligned}
    \|u_h^k - u^k\|_{(\alpha,\beta)}^2 &\le C(\|\theta^k\|_{(\alpha,\beta)}^2 
    + \|\rho^k\|_{(\alpha,\beta)}^2) \\
    &\le C g_k \exp(C'TM_2^2) + C h^{2\mu-2\lambda}\|u(t_k)\|^2_{\mu} \\
    &\le C\exp(C'TM_2^2) \Big( \|u^0_h - u(t_0)\|^2_{\lambda}  + 
    \tau^2 \left( \|u_{tt}\|^2_{0,0} + M_2^2 \|u_t\|^2_{0,0} \right)  \\
    & + h^{2\mu-2\lambda} \big( \|u(t_0)\|^2_{\mu} + \|u(t_k)\|^2_{\mu}+ \|u_t\|_{0,\mu}^2 
    + T M_2^2 \max_{0 \le t \le T} \|u(t)\|_{\mu}^2  \big) \Big)
  \end{aligned}
\end{equation}
When choosing the interpolation as initial value of $u$ at time $t_0$, 
i.e.\ $u_h^0 = I_h u(t_0)$, the following inequality holds
\begin{equation}
  \begin{aligned}
    \|u_h^k - u^k\|_{(\alpha,\beta)}^2 
    &\le C\exp(C'TM_2^2) \Big(  
    \tau^2 \left( \|u_{tt}\|^2_{0,0} + M_2^2 \|u_t\|^2_{0,0} \right)  \\
    & + h^{2\mu-2\lambda} \big( \|u(t_0)\|^2_{\mu} + \|u(t_k)\|^2_{\mu}+ \|u_t\|_{0,\mu}^2 
    + T M_2^2 \max_{0 \le t \le T} \|u(t)\|_{\mu}^2  \big) \Big).
  \end{aligned}
\end{equation}
We finish the analysis of convergence of BEGM. 
The convergence theorem is shown below as a conclusion.
\begin{theorem}[convergence]\label{the:convergence}
  Suppose that 
  \begin{enumerate}
    \renewcommand{\labelenumi}{(\theenumi)}
    \item $u_h^n$ are solutions of Eq.~\eqref{eq:scheme1},
    \item {$F(x) \in C^1(\Theta)$, $|F(x)| \le M_1$, and
      $|F'(x)| \le M_2$}. 
    \item The exact solution $u\in L^\infty(0,T;H_0^{\mu} (\Omega))$, 
      $u_{t} \in L^2(0, T;H_0^{\mu})$, and $u_{tt} \in L^2(0, T;L^2(\Omega))$, 
      where $\lambda < \mu \le s+1$.
  \end{enumerate}
  Also assume $\tau < \frac{\mathcal{B}}{20M_2^2}$. Then we have
  \begin{equation}
    \begin{aligned}
      \|u_h^k - &u^k\|_{(\alpha,\beta)}^2 \le C \tau^2 
      \left( \|u_{tt}\|^2_{0,0} + \|u_t\|^2_{0,0} \right) \\
      & + C h^{2\mu-2\lambda} 
      \big( \|u(t_0)\|^2_{\mu} + \|u(t_k)\|^2_{\mu}+ \|u_t\|_{0,\mu}^2 
        + T M_2^2 \max_{0 \le t \le T} \|u(t)\|_{\mu}^2  \big).
    \end{aligned}
  \end{equation}
\end{theorem}
\mOne{
\begin{remark}
  In Theorem~\ref{the:convergence}, $\mu$ is related with smoothness of the exact solution
  $u$ and with the element type used in simulation. If we use linear
  triangular element in numerical examples, the degree of polynomials in space 
  $V_h$ is at most $1$, i.e.\ $s=1$, then $\lambda < \mu \le 2$. 
  When the exact solution $u$ is good enough, we can get $\mu = 2$,
  then the space convergence order is $2 - \lambda$.

\end{remark}}

\section{Numerical examples}\label{sec:examples}
In this section, we consider some numerical examples to demonstrate the
effectiveness of our theoretical analysis.  Here, linear triangular element is
used.
\begin{example}\label{example:rect}
  Consider the following fractional problem
  \begin{equation}\label{eq:example1}
    \left\{
    \begin{aligned}
      & \frac{\partial u}{\partial t} = 
         K_x\frac{\partial^{2\alpha}u}{\partial|x|^{2\alpha}} + 
         K_y\frac{\partial^{2\beta}u}{\partial|y|^{2\beta}} + F(u) + f(x, y, t),\\
      & u(x, y, 0) = \varphi(x, y), \quad (x,y) \in \Omega, \\
    & u(x, y, t) = 0, \quad (x,y,t) \in \partial\Omega \times (0, T],
    \end{aligned}
    \right.
  \end{equation}
  where $\Omega=(0,1)\times(0,1)$, $F(u) = -u^2$, 
  $\varphi(x, y) = 10x^2{(1-x)}^2y^2{(1-y)}^2$, and
  \begin{equation}
    \begin{aligned}
      f(x,y,t) =& -10e^{-t}x^2{(1-x)}^2y^2{(1-y)}^2 + 100e^{-2t}x^4{(1-x)}^4y^4{(1-y)}^4 \\
      &+ 10K_xc_\alpha e^{-t}y^2(1-y)^2\big(g(x, \alpha) + g(1-x, \alpha)\big) \\
      &+ 10K_yc_\beta e^{-t}x^2(1-x)^2\big(g(y, \beta) + g(1-y, \beta)\big), \\
    \end{aligned}
  \end{equation}
  \begin{equation}
    g(x,\alpha) = \frac{\Gamma(5)}{\Gamma(5-2\alpha)}x^{4-2\alpha} 
        - \frac{2\Gamma(4)}{\Gamma(4-2\alpha)}x^{3-2\alpha} 
        + \frac{\Gamma(3)}{\Gamma(3-2\alpha)}x^{2-2\alpha}.
  \end{equation}
  The exact solution of Eq.~\eqref{eq:example1} is
  $$u(x,y,t)=10e^{-t}x^2{(1-x)}^2y^2{(1-y)}^2.$$

  In this example, we take $K_x=K_y=1$, $T = 1$ and compute the numerical results
  by different $\alpha$ and $\beta$.  Two meshes used in computation are shown in
  Fig.~\ref{fig:rectanglemesh}. The results are given in
  Table~\ref{table:rect}.  As we use linear triangular element, the spatial
  convergence order should be $2-\max\{\alpha,\beta\}$. By \mMy{examining} the rates of
  convergence shown in Table~\ref{table:rect}, we notice that the spatial
  convergence orders fit those proved in Theorem~\ref{the:convergence}.  
  \mOne{The
  temporal convergence order for different norms are given in
  Table~\ref{table:rect1}.  As we can see, the convergence order of norm
  $||\cdot||_{(\alpha, \beta)}$ is close to 1.  
  These results agree with the result of theoretical analysis. }
  Besides, we plot numerical and exact solutions at $T=1$ with $h  \approx 1/40$
  in Fig.~\ref{fig:rectangleEN}, which indicates that the numerical result is a good 
  approximation of exact solution.

  \begin{figure}
    \centering
    \includegraphics[width=0.4\textwidth]{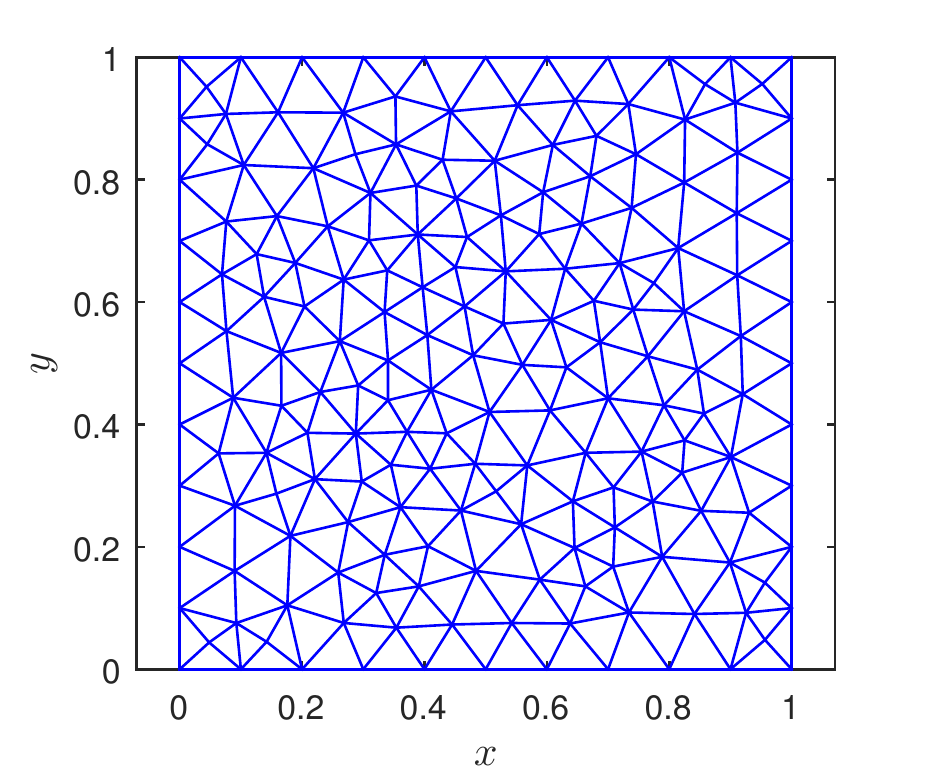} \hskip15pt
    \includegraphics[width=0.4\textwidth]{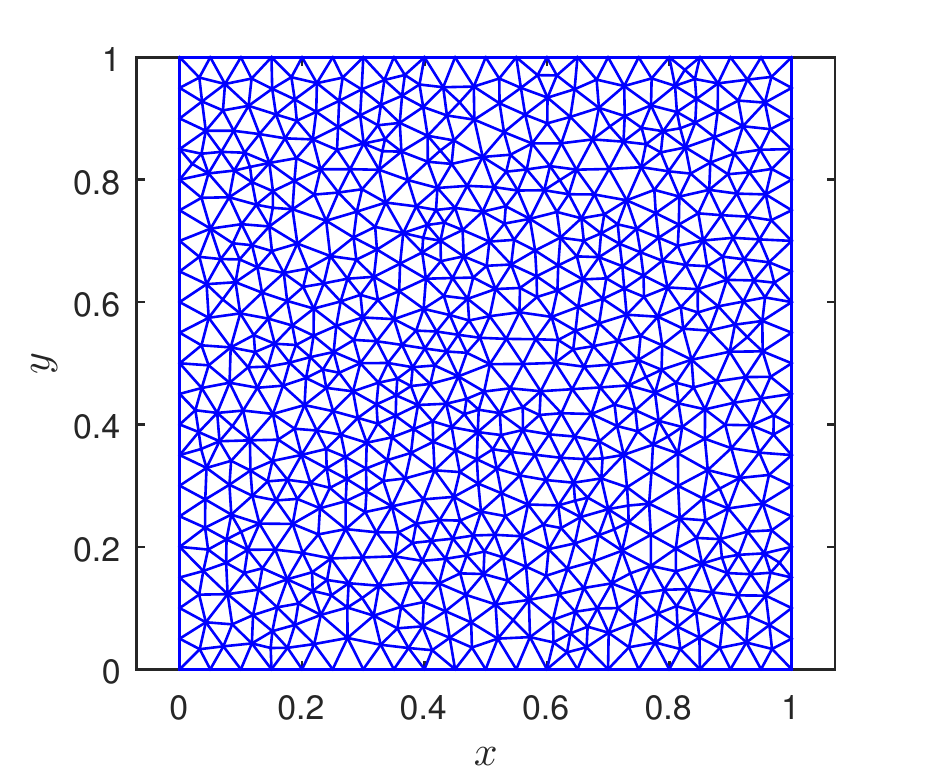}
    \caption{Meshes on rectangle domain with $h \approx 1/10$ 
      and $h\approx1/20$.}\label{fig:rectanglemesh}
  \end{figure}

  \begin{table}
    \caption{Errors and space convergence orders of BEGM for
      Example~\ref{example:rect} ($\tau = h^2$).}\label{table:rect}
    \footnotesize \centering
    \begin{tabular}{cccccccc}
      \toprule
    &   h &     $L^2$ error &     Order &   $L^\infty$ error &   Order &
        $L^{(\alpha,\beta)}$ error & Order \\
      \midrule \multirow{5}{4em}{$\alpha=0.8$\\$\beta=0.8$}
    &  1/5  &     5.01e-04 &         &     6.19e-04 &         &     1.18e-03 &         \\
    & 1/10  &     1.43e-04 &    1.81 &     1.49e-04 &    2.06 &     4.86e-04 &    1.29 \\
    & 1/20  &     3.91e-05 &    1.87 &     6.25e-05 &    1.25 &     2.20e-04 &    1.14 \\
    & 1/40  &     1.04e-05 &    1.91 &     1.95e-05 &    1.68 &     1.05e-04 &    1.06 \\
      \midrule \multirow{5}{4em}{$\alpha=0.95$\\$\beta=0.95$}
    &  1/5   &     5.50e-04 &         &     8.57e-04 &         &     2.01e-03 &         \\
    & 1/10   &     1.59e-04 &    1.79 &     1.86e-04 &    2.21 &     9.02e-04 &    1.16 \\
    & 1/20   &     4.33e-05 &    1.87 &     6.97e-05 &    1.41 &     4.28e-04 &    1.08 \\
    & 1/40   &     1.10e-05 &    1.98 &     2.07e-05 &    1.75 &     1.90e-04 &    1.17 \\
      \midrule \multirow{5}{4em}{$\alpha=0.8$\\$\beta=0.75$}
    &  1/5  &     4.97e-04 &         &     6.02e-04 &         &     1.08e-03 &         \\
    & 1/10  &     1.42e-04 &    1.81 &     1.47e-04 &    2.03 &     4.50e-04 &    1.26 \\
    & 1/20  &     3.94e-05 &    1.85 &     7.46e-05 &    0.98 &     2.10e-04 &    1.10 \\
    & 1/40  &     1.06e-05 &    1.89 &     2.06e-05 &    1.86 &     1.03e-04 &    1.03 \\
      \bottomrule
    \end{tabular} 
  \end{table}
  \begin{table}
    \caption{Errors and temporal convergence orders of BEGM for Example~\ref{example:rect} 
      with $\alpha = 0.85$, $\beta = 0.85$ ($h=\tau$).}\label{table:rect1}
    \footnotesize \centering
    \mOne{
      \begin{tabular}{ccccccc}
  \toprule
       $\tau$ &     $L^2$ error &     Order &   $L^\infty$ error &   Order & 
       $L^{(\alpha,\beta)}$ error & Order \\
  \midrule
   1/5  &4.72e-04 &         &5.60e-04 &         &1.31e-03 &         \\
  1/10  &1.38e-04 &    1.77 &1.69e-04 &    1.73 &5.84e-04 &    1.17 \\
  1/20  &3.76e-05 &    1.88 &6.91e-05 &    1.29 &2.70e-04 &    1.11 \\
  1/40  &9.76e-06 &    1.94 &2.09e-05 &    1.73 &1.22e-04 &    1.15 \\
  \bottomrule
\end{tabular}
    }
  \end{table}
  \begin{figure}
    \centering
    \includegraphics[width=0.4\textwidth]{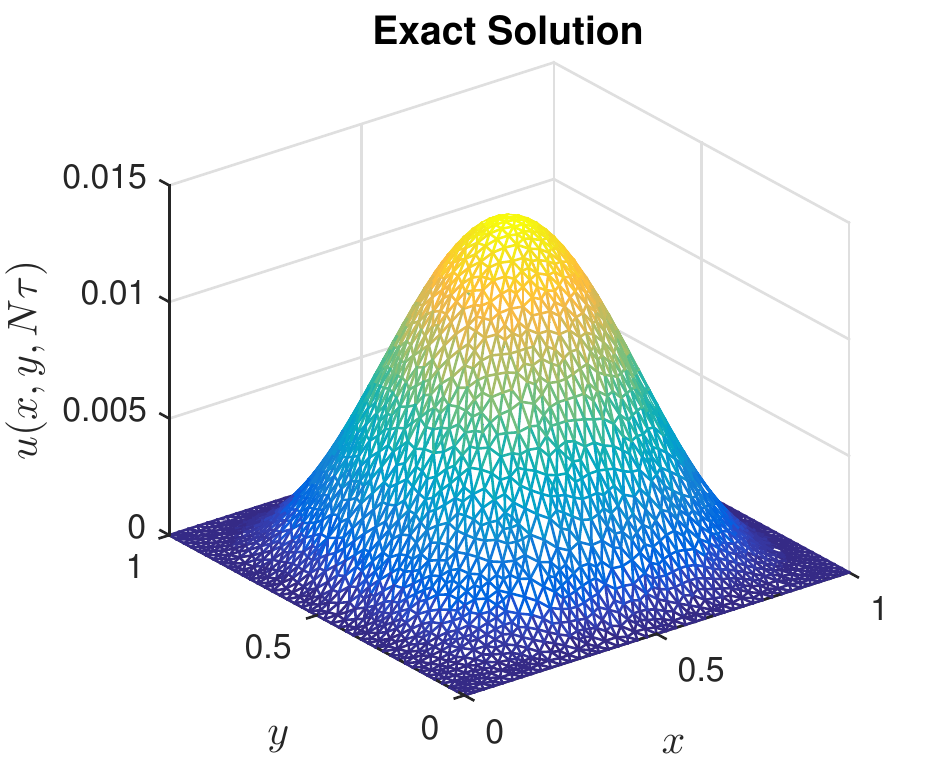} \hskip15pt
    \includegraphics[width=0.4\textwidth]{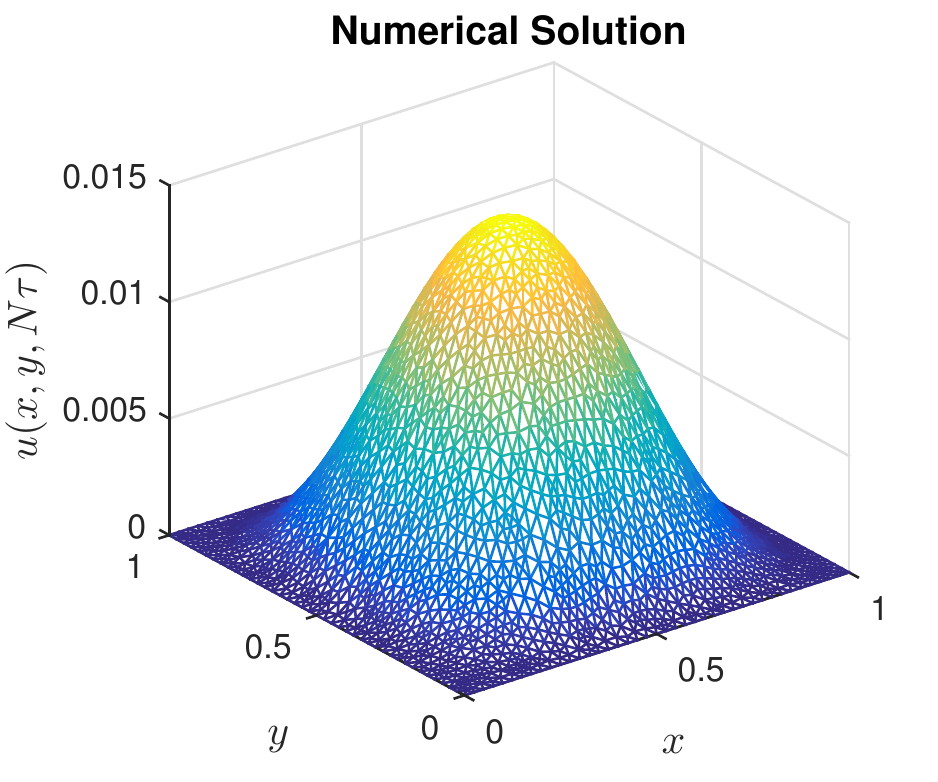}
    \caption{The exact solution and numerical approximation when $T=1$,
      $h\approx1/40$ on rectangle domain.}\label{fig:rectangleEN}
  \end{figure}
\end{example}

\begin{example}\label{example:poly}
  Consider the following fractional problem
  \begin{equation}\label{eq:poly}
    \left\{
    \begin{aligned}
      & \frac{\partial u}{\partial t} = 
         K_x\frac{\partial^{2\alpha}u}{\partial|x|^{2\alpha}} + 
         K_y\frac{\partial^{2\beta}u}{\partial|y|^{2\beta}} + F(u) + f(x, y, t),\\
      & u(x, y, 0) = \varphi(x, y), \quad (x,y) \in \Omega, \\
    & u(x, y, t) = 0, \quad (x,y,t) \in \partial\Omega \times (0, T],
    \end{aligned}
    \right.
  \end{equation}
  where the domain $\Omega$ is shown in Fig.~\ref{fig:poly}, $F(u) = -u^2$, 
  $\varphi(x, y) = 1000x^2{(1-x)}^2(x+y-1.5)^2y^2{(1-y)}^2$, and
  \begin{equation}
    \begin{aligned}
      f(x,y,t) =& -1000e^{-t}x^2{(1-x)}^2(x+y-1.5)^2y^2{(1-y)}^2 \\
      &+ 10^6e^{-2t}x^4{(1-x)}^4(x+y-1.5)^4y^4{(1-y)}^4 \\
      &+ 1000e^{-t}K_xc_\alpha y^2(1-y)^2g(x,y, \alpha) \\
      &+1000e^{-t}K_yc_\beta x^2(1-y)^2g(y,x, \beta),
    \end{aligned}
  \end{equation}
  \begin{equation}
    g(x,y, \alpha) = \left\{
    \begin{aligned}
      &g_0(x, 1.5-y, \alpha) + g_1(1-x, 1.5-y, \alpha), && y \le 0.5, \\
      &g_0(x, 1.5-y, \alpha) + g_2(1.5-y-x, 1.5-y, \alpha), && y > 0.5, 
    \end{aligned}\right.
  \end{equation}
  \begin{equation}
    \begin{aligned}
    g_0(x, y, \alpha) =&\ \frac{2 y^2 x^{2-2 \alpha }}{\Gamma (3-2 \alpha )} 
    -\frac{12 \left(y^2-y\right) x^{3-2 \alpha }}{\Gamma (4-2 \alpha )} 
    +\frac{24 \left(y^2+4 y+1\right) x^{4-2 \alpha }}{\Gamma (5-2 \alpha )}  \\
    &-\frac{240 (y+1) x^{5-2 \alpha }}{\Gamma (6-2 \alpha )}
    +\frac{720 x^{6-2 \alpha }}{\Gamma (7-2 \alpha )},
    \end{aligned}
  \end{equation}
  \begin{equation}
    \begin{aligned}
      g_1(x,y,\alpha) = &\ 
      \frac{2 \left(y^2-2 y+1\right) x^{2-2 \alpha }}{\Gamma (3-2 \alpha )}
      -\frac{12 \left(y^2-3 y+2\right) x^{3-2 \alpha }}{\Gamma (4-2 \alpha )} \\
      &+\frac{24 \left(y^2-6 y+6\right) x^{4-2 \alpha }}{\Gamma (5-2 \alpha )} 
      +\frac{240 (y-2) x^{5-2 \alpha }}{\Gamma (6-2 \alpha )} 
      + \frac{720 x^{6-2 \alpha }}{\Gamma (7-2 \alpha )},
    \end{aligned}
  \end{equation}
  \begin{equation}
    \begin{aligned}
      g_2(x,y,\alpha) = &\ 
      \frac{2 \left(y^4-2 y^3+y^2\right) x^{2-2 \alpha }}{\Gamma (3-2 \alpha )} 
      - \frac{12 \left(2 y^3-3 y^2+ y\right) x^{3-2 \alpha }}{\Gamma (4-2 \alpha )} \\
      &+\frac{24 \left(6 y^2-6 y+1\right) x^{4-2 \alpha }}{\Gamma (5-2 \alpha )}
      +\frac{240 (1-2 y) x^{5-2 \alpha }}{\Gamma (6-2 \alpha )}
      +\frac{720 x^{6-2 \alpha }}{\Gamma (7-2 \alpha )},
    \end{aligned}
  \end{equation}
  The exact solution of Eq.~\eqref{eq:poly} is
  $$u(x,y,t)=1000e^{-t}x^2{(1-x)}^2(x+y-1.5)^2y^2{(1-y)}^2.$$

  \begin{figure}
    \centering
    \includegraphics[width=0.4\textwidth]{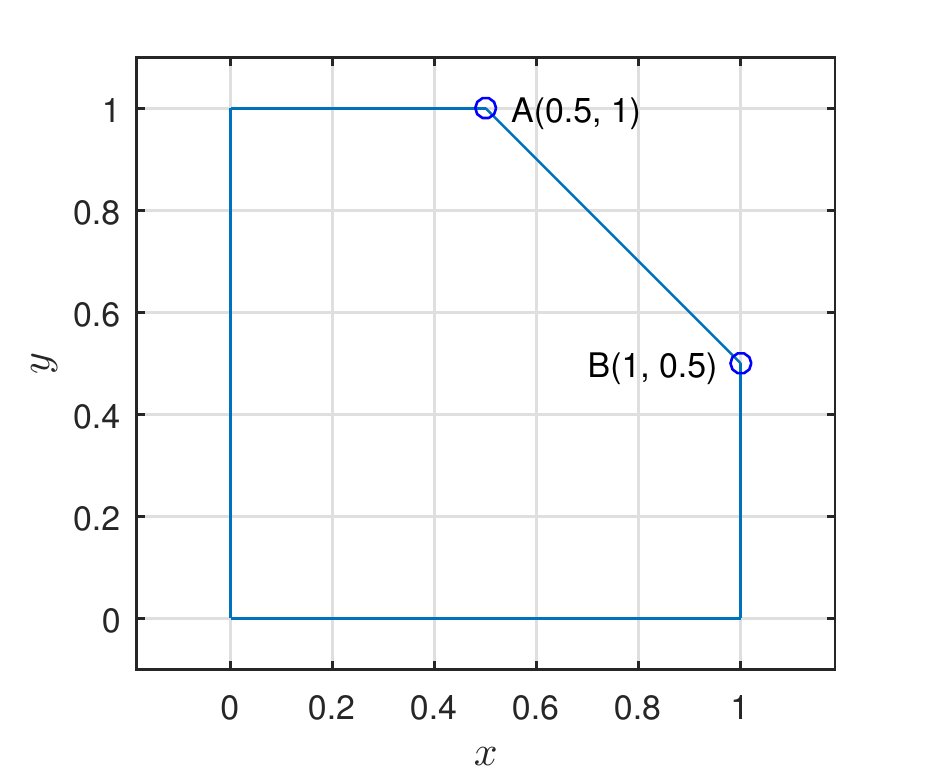} \hskip15pt
    \includegraphics[width=0.4\textwidth]{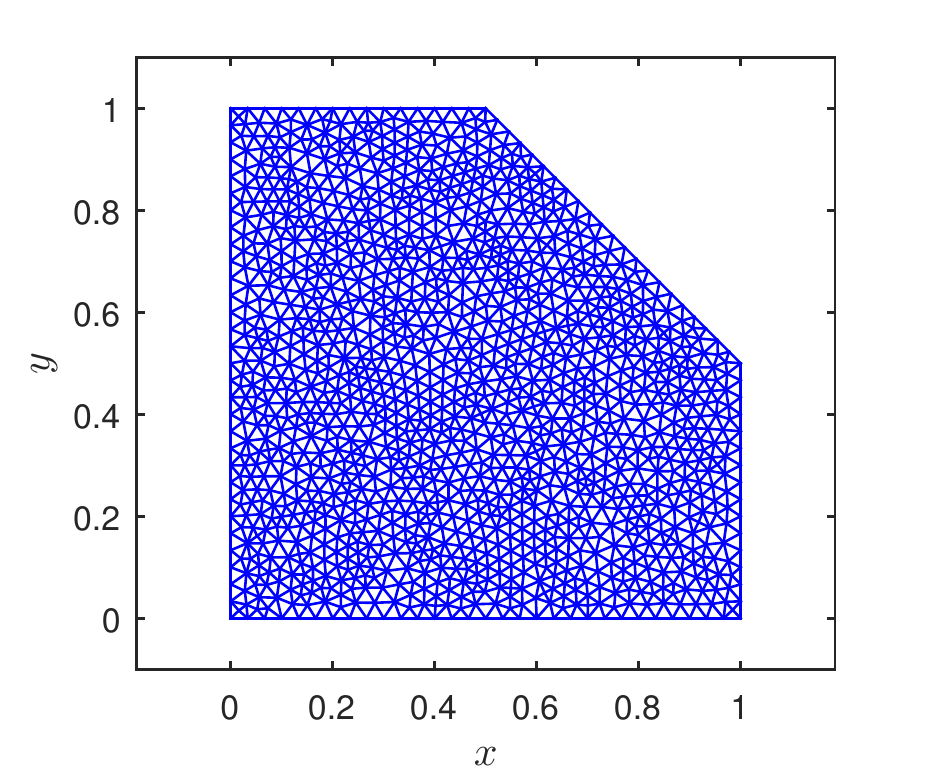}
    \caption{Poly domain and mesh with $h \approx 1/30$.}\label{fig:poly}
  \end{figure}
  \begin{table}
    \caption{Errors and spatial convergence orders of BEGM for
      Example~\ref{example:poly}($\tau=h^2$).}\label{table:poly1}
    \footnotesize \centering
    \mOne{
    \begin{tabular}{cccccccc}
      \toprule
      &   h &     $L^2$ error &     Order &   $L^\infty$ error &   Order & 
      $L^{(\alpha,\beta)}$ error & Order \\
      \midrule \multirow{5}{4em}{$\alpha=0.8$\\ $\beta = 0.8$}
        &  1/5     &     2.26e-02 &         &     2.66e-02 &         &     6.77e-02 &         \\
        & 1/10     &     7.76e-03 &    1.54 &     1.30e-02 &    1.04 &     3.50e-02 &    0.95 \\
        & 1/20     &     2.24e-03 &    1.79 &     5.37e-03 &    1.27 &     1.54e-02 &    1.19 \\
        & 1/30     &     9.47e-04 &    2.13 &     2.84e-03 &    1.57 &     9.32e-03 &    1.23 \\
      \midrule \multirow{5}{4em}{$\alpha=0.95$\\ $\beta = 0.95$}
     &  1/5     &     2.51e-02 &         &     3.73e-02 &         &     1.20e-01 &         \\
     & 1/10     &     8.41e-03 &    1.58 &     1.41e-02 &    1.40 &     6.45e-02 &    0.90 \\
     & 1/20     &     2.46e-03 &    1.77 &     6.50e-03 &    1.12 &     3.05e-02 &    1.08 \\
     & 1/30     &     1.03e-03 &    2.14 &     3.55e-03 &    1.49 &     1.89e-02 &    1.17 \\
       \midrule \multirow{5}{4em}{$\alpha=0.8$\\ $\beta = 0.7$}
     &  1/5     &     2.24e-02 &         &     2.59e-02 &         &     5.32e-02 &         \\
     & 1/10     &     7.67e-03 &    1.55 &     1.12e-02 &    1.21 &     2.81e-02 &    0.92 \\
     & 1/20     &     2.27e-03 &    1.76 &     4.50e-03 &    1.31 &     1.25e-02 &    1.17 \\
     & 1/30     &     9.67e-04 &    2.10 &     2.30e-03 &    1.65 &     7.62e-03 &    1.23 \\
      \bottomrule
    \end{tabular}}
  \end{table}
  \begin{table}
    \caption{Errors and temporal convergence orders of BEGM for
      Example~\ref{example:poly} with $\alpha = 0.8$, $\beta = 0.9$
      ($h=\tau$).}\label{table:poly2}
    \footnotesize \centering
    \mOne{
    \begin{tabular}{ccccccc}
    \toprule
    $\tau$ &     $L^2$ error &     Order &   $L^\infty$ error &   Order & 
     $L^{(\alpha,\beta)}$ error & Order \\
    \midrule
      1/10 &7.98e-03 &         &1.49e-02 &         &4.92e-02 &         \\
      1/20 &2.25e-03 &    1.83 &6.75e-03 &    1.14 &2.28e-02 &    1.11 \\
      1/30 &9.64e-04 &    2.09 &3.85e-03 &    1.38 &1.44e-02 &    1.14 \\
      1/40 &5.71e-04 &    1.82 &2.48e-03 &    1.54 &1.05e-02 &    1.08 \\
    \bottomrule
  \end{tabular}}
\end{table}
  In this test, we take $K_x=1, K_y=2$, $T = 1$.
   We compute the $L^2(\Omega)$ errors, the $L^\infty(\Omega)$ errors,
  the $||\cdot||_{(\alpha, \beta)}$ errors, and the spatial convergence orders.
  Fig.~\ref{fig:poly} shows the mesh used in computation with $h\approx 1/16$.
  The results are given in Table~\ref{table:poly1}. 
  As shown in the table, 
  the numerical results agree well with Theorem~\ref{the:convergence}.
  We also give the temporal convergence order in Table~\ref{table:poly2}. 
  As we can see from the table, the numerical orders of $L^{(\alpha,\beta)}$ error agree well with 
  the analysis result.
  Also, we plot numerical and exact solutions at $T=1$ with $h  \approx 1/30$
  in Fig.~\ref{fig:poly1}, which appears that the numerical result is a good 
  approximation of the exact solution.

\begin{figure}
  \centering
  \includegraphics[width=0.4\textwidth]{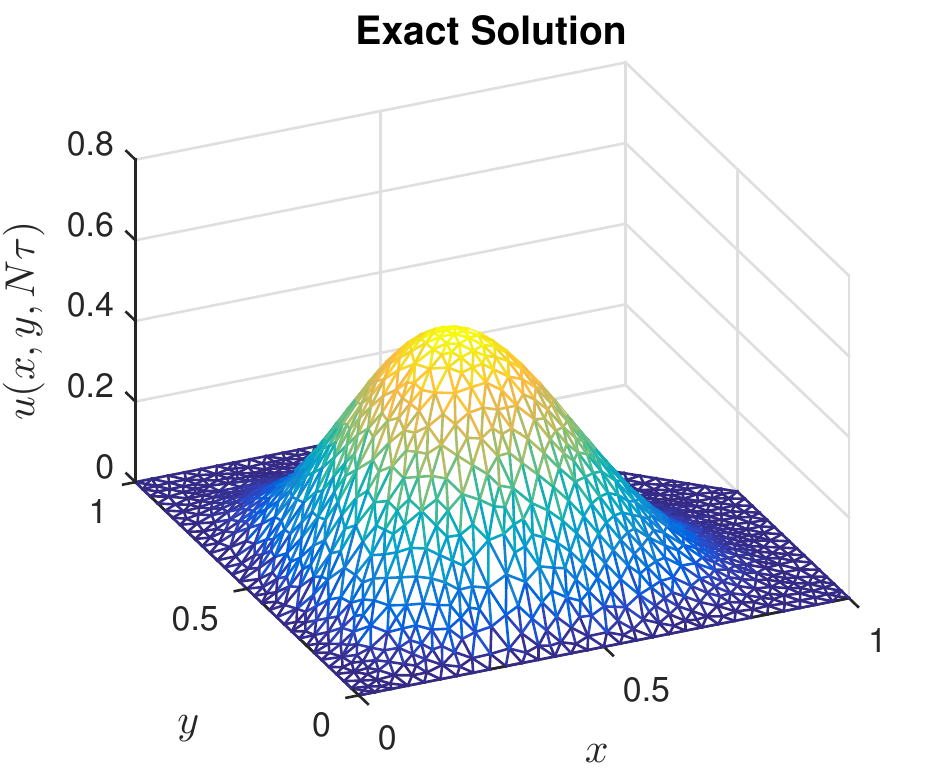} \hskip15pt
  \includegraphics[width=0.4\textwidth]{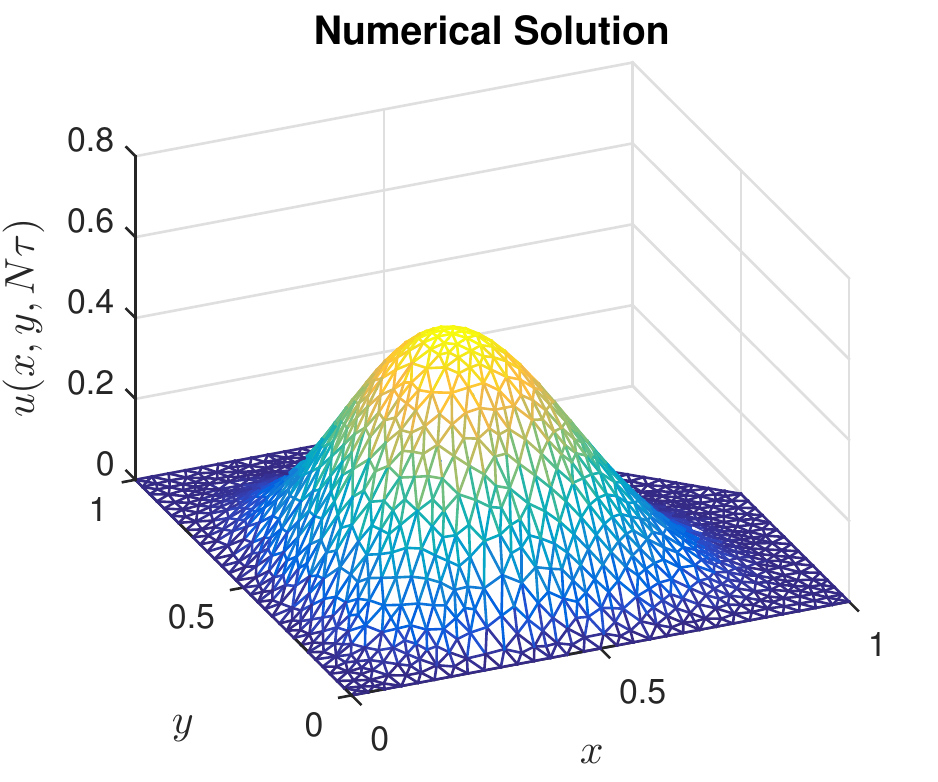}
  \caption{The exact solution and numerical approximation when $T=1$, 
  $h\approx1/30$ on poly domain.}\label{fig:poly1}
\end{figure}

\end{example}

\begin{example}\label{example:ell}
  Consider the following fractional model
  \begin{equation}\label{eq:ell}
    \left\{
    \begin{aligned}
      & \frac{\partial u}{\partial t} = 
         K_x\frac{\partial^{2\alpha}u}{\partial|x|^{2\alpha}} + 
         K_y\frac{\partial^{2\beta}u}{\partial|y|^{2\beta}} + F(u) + f(x, y, t),\\
      & u(x, y, 0) = \varphi(x, y), \quad (x,y) \in \Omega, \\
    & u(x, y, t) = 0, \quad (x,y,t) \in \partial\Omega \times (0, T],
    \end{aligned}
    \right.
  \end{equation}
  where $\Omega=\{(x,y):\frac{x^2}{a^2}+\frac{y^2}{b^2}<1\}$, $a>0$, $b>0$, $F(u) = -u^2$, 
  $\varphi(x, y) = 100\left(b^2 x^2+a^2 y^2-a^2 b^2\right)^2$, and
  \begin{equation}
    \begin{aligned}
      f(x,y,t) =& -100e^{-t}{(x^2+y^2-r^2)}^2 + 10^4e^{-2t}{(x^2+y^2-r^2)}^4 \\
        &+100K_xc_\alpha e^{-t}h(x+a\sqrt{1-{y^2}/{b^2}},      -\sqrt{1-{y^2}/{b^2}}) \\ 
        &+ 100K_xc_\alpha e^{-t}h(a\sqrt{1-{y^2}/{b^2}}-x, \sqrt{1-{y^2}/{b^2}})  \\
        &+100K_yc_\beta e^{-t} g(-\sqrt{1-{x^2}/{a^2}}, y+b\sqrt{1-{x^2}/{a^2}})  \\
        &+ 100K_yc_\beta e^{-t} g(\sqrt{1-{x^2}/{a^2}},   b\sqrt{1-{x^2}/{a^2}}-y)  \\
    \end{aligned}
  \end{equation}
  \begin{equation}
    h(x,y) = \frac{8 a^2 b^4 y^2 x^{2-2 \alpha }}{\Gamma (3-2 \alpha )}+\frac{24 a b^4
   y x^{3-2 \alpha }}{\Gamma (4-2 \alpha )}+\frac{24 b^4 x^{4-2 \alpha
   }}{\Gamma (5-2 \alpha )}
  \end{equation}
  \begin{equation}
    g(x,y) = \frac{8 a^4 b^2 x^2 y^{2-2 \beta }}{\Gamma (3-2 \beta )}+\frac{24 a^4 b
   x y^{3-2 \beta }}{\Gamma (4-2 \beta )}+\frac{24 a^4 y^{4-2 \beta
   }}{\Gamma (5-2 \beta )}
  \end{equation}
  The exact solution to Eq.~\eqref{eq:ell} is
  $$u(x,y,t)=100e^{-t}{(b^2 x^2+a^2 y^2-a^2 b^2)}^2.$$ 
  
  In this example, set $a = 1/2$, $b = 3/4$, then the domain is an ellipse. 
  Choose parameters $K_x=2$, $K_y=2$, $T=1$.  
  Table~\ref{table:ell} shows the spatial convergence orders. 
  As $\alpha$, $\beta$ increase, the convergence order of $||\cdot||_{(\alpha, \beta)}$ 
  errors decrease and the orders are close to $2-\max\{\alpha, \beta\}$. 
  These results agree well with what we have proved in Theorem~\ref{the:convergence}.
  For the temporal direction, we give the results in Table~\ref{table:ell1}.
  Fig.~\ref{fig:ellmesh} presents the computational domain
  and the mesh used in this example with $h\approx 1/30$
  and Fig.~\ref{fig:ell1} gives a comparison between exact solution and numerical solution.
  These results shows that our algorithm also works well in elliptical domain.
\begin{figure}
  \centering
   \includegraphics[width=0.4\textwidth]{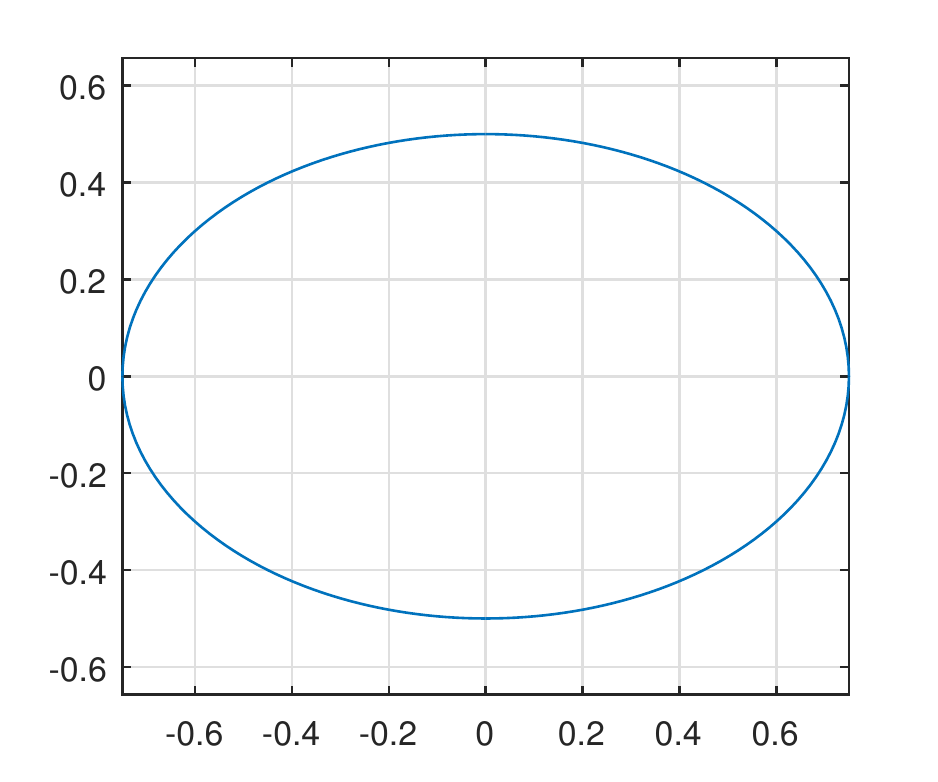}
    \hskip15pt
    \includegraphics[width=0.4\textwidth]{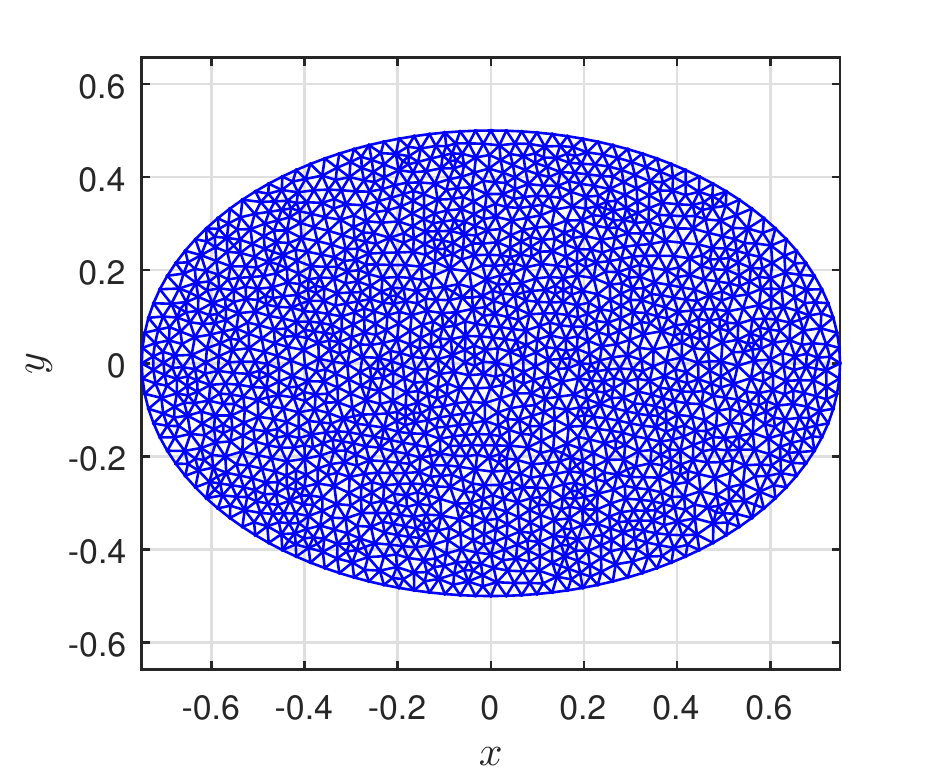}
   \caption{Elliptical  domain and mesh with  $h\approx1/30$.}\label{fig:ellmesh}
\end{figure}
\begin{figure}
  \centering
    \includegraphics[width=0.4\textwidth]{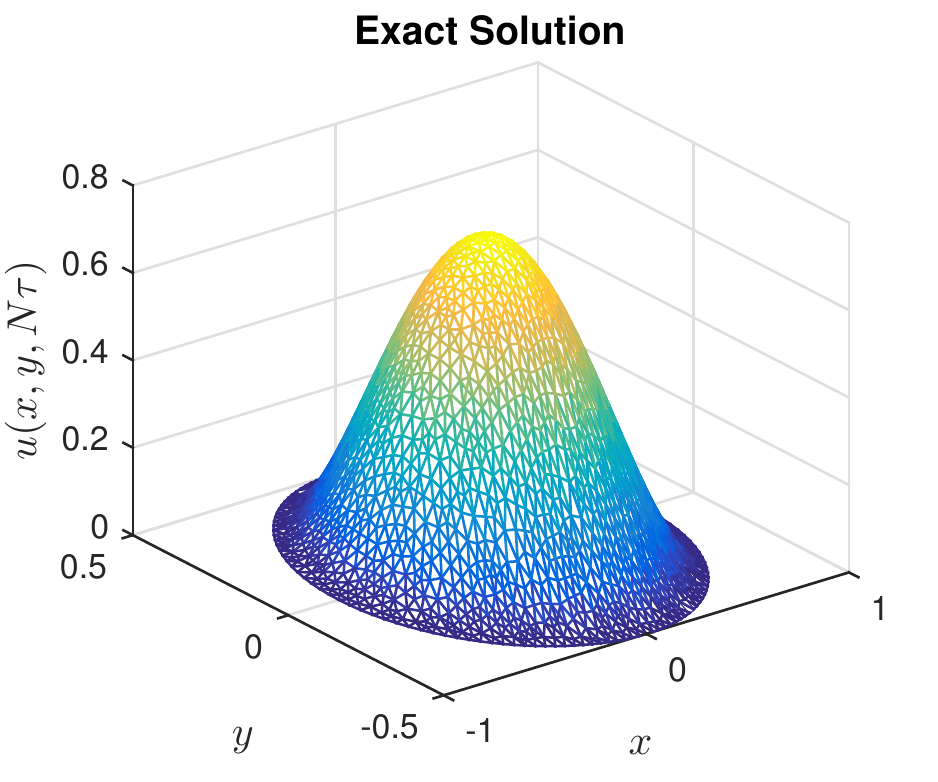}
    \hskip15pt
    \includegraphics[width=0.4\textwidth]{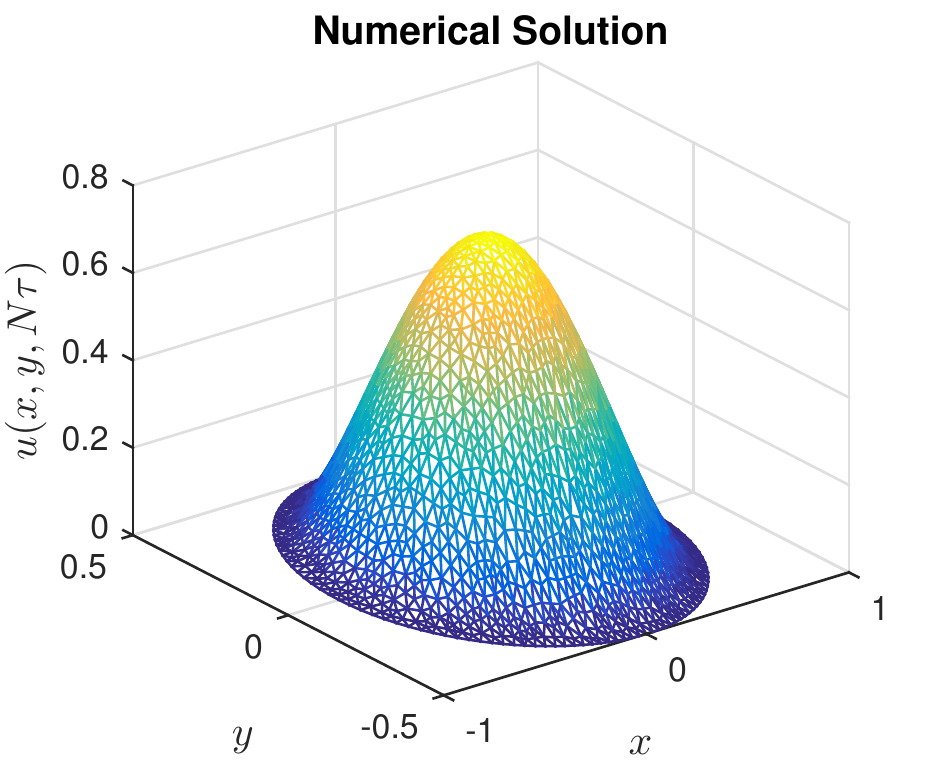}
  \caption{The exact solution and numerical approximation when $T=1$, 
  $h\approx1/30$ on elliptical domain.}\label{fig:ell1}
\end{figure}
  \begin{table}
  \caption{Errors and space convergence orders of BEGM for Example~\ref{example:ell}
  ($\tau = h^2$).}\label{table:ell}
  \footnotesize
  \centering
  \begin{tabular}{cccccccc}
    \toprule
      &   h &     $L^2$ error &     Order &   $L^\infty$ error &   Order & $L^{(\alpha,\beta)}$ error & Order \\
    \midrule \multirow{5}{4em}{$\alpha=0.85$\\$\beta=0.85$}
  &       1/5    &     2.80e-02 &         &     2.45e-02 &         &     9.63e-02 &         \\
  &      1/10    &     7.73e-03 &    1.85 &     8.93e-03 &    1.46 &     4.38e-02 &    1.14 \\
  &      1/20    &     2.02e-03 &    1.93 &     3.14e-03 &    1.51 &     1.89e-02 &    1.21 \\
  &      1/30    &     9.15e-04 &    1.96 &     1.55e-03 &    1.75 &     1.21e-02 &    1.11 \\
    \midrule
    \multirow{5}{4em}{$\alpha=0.95$\\$\beta=0.95$}
  &       1/5 &     3.06e-02 &         &     3.08e-02 &         &     1.40e-01 &         \\
  &      1/10 &     8.32e-03 &    1.88 &     9.52e-03 &    1.70 &     6.68e-02 &    1.07 \\
  &      1/20 &     2.18e-03 &    1.93 &     3.43e-03 &    1.47 &     2.99e-02 &    1.16 \\
  &      1/30 &     9.79e-04 &    1.97 &     1.70e-03 &    1.74 &     1.91e-02 &    1.10 \\
    \midrule
    \multirow{5}{4em}{$\alpha=0.8$\\$\beta=0.75$}
   &       1/5 &     2.74e-02 &         &     2.26e-02 &         &     7.97e-02 &         \\
   &      1/10 &     7.69e-03 &    1.83 &     8.87e-03 &    1.35 &     3.36e-02 &    1.25 \\
   &      1/20 &     2.03e-03 &    1.92 &     3.11e-03 &    1.51 &     1.48e-02 &    1.18 \\
   &      1/30 &     9.37e-04 &    1.91 &     1.71e-03 &    1.48 &     1.01e-02 &    0.94 \\
    \bottomrule
  \end{tabular} 
\end{table}

\begin{table}
\caption{Errors and temporal convergence orders of BEGM for Example~\ref{example:ell} with $\alpha = 0.85$, $\beta = 0.85$ ($h=\tau$).}\label{table:ell1}
\footnotesize
\centering
\mOne{
  \begin{tabular}{ccccccc}
  \toprule
     $\tau$ &     $L^2$ error &     Order &   $L^\infty$ error &   Order & $L^{(\alpha,\beta)}$ error & Order \\
  \midrule
   1/5 & 1.04e-02 &          & 8.70e-03 &          & 3.63e-02 &          \\
  1/10 & 2.95e-03 &     1.82 & 3.65e-03 &     1.25 & 1.73e-02 &     1.07 \\
  1/20 & 7.69e-04 &     1.94 & 1.29e-03 &     1.51 & 7.56e-03 &     1.19 \\
  1/30 & 3.40e-04 &     2.01 & 6.33e-04 &     1.75 & 4.83e-03 &     1.10 \\
  \bottomrule
\end{tabular}
}
\end{table}
\end{example}

\begin{example}
  Consider the fractional FitzHugh-Nagumo problem
  \begin{equation}\label{eq:exampleFN}
    \left\{
    \begin{aligned}
      & \frac{\partial u}{\partial t} = 
         K_x\frac{\partial^{2\alpha}u}{\partial|x|^{2\alpha}} + 
         K_y\frac{\partial^{2\beta}u}{\partial|y|^{2\beta}} + u(1-u)(u-\mu) - w,\\
      &\frac{\partial w}{\partial t} = \epsilon(\lambda u - \gamma w - \delta),\quad (x,y,t) \in \partial\Omega \times T, \\
      & \Omega=\{(x,y): (x-r)^2+(y-r)^2<r^2\},
    \end{aligned}
    \right.
  \end{equation}
  where $r=1.25$, $\mu = 0.1$, $\epsilon = 0.01$,
  $\lambda = 0.5$, $  \gamma = 0.1$, $\delta = 0$. The initial conditions are taken as
  \begin{equation}
    u(x, y, 0) = \left\{
      \begin{aligned}
        1, &\quad x < r, y < r \\
        0, &\quad elsewhere
      \end{aligned}\right.,
    \quad
    w(x, y, 0) = \left\{
      \begin{aligned}
        0.1, &\quad y\ge r \\
        0,   &\quad elsewhere
      \end{aligned}\right.,
  \end{equation}
  \mOne{and the boundary conditions are homogeneous.}
  For this coupled differential equation, we first solve the fractional Riesz space 
  nonlinear equation by given $w$ and $u$, then solve the ordinary differential equation
  with $w$ and new $u$ at each time step. 
  
  The simulation results with $K_x=K_y=0.0001$ and $K_x = 4K_y=0.00025$ at
  $T=1000$ are show in Fig.~\ref{fig:fhn_1} and Fig.~\ref{fig:fhn_2},
  respectively.  From Fig.~\ref{fig:fhn_1}, we notice that as fractional
  orders $\alpha,\beta$ decrease, the wave travels more slowly.
  Fig.~\ref{fig:fhn_2} shows anisotropic diffusion with different
  coefficients in spatial dimensions. In this situation, the wave behave
  different \mMy{velocities} in spatial directions. These results reported
  in \cite{Liu2015} are consistent with our results.  In addition, Zeng
  et al.\ \cite{Zeng2014} solved this problem in rectangle domain, and their
results are similar with our results.

\begin{figure}
  \centering
  \subfloat[$\alpha=0.75,\beta=0.75$]{
    \begin{minipage}[t]{0.3\textwidth}
      \includegraphics[width=0.9\textwidth]{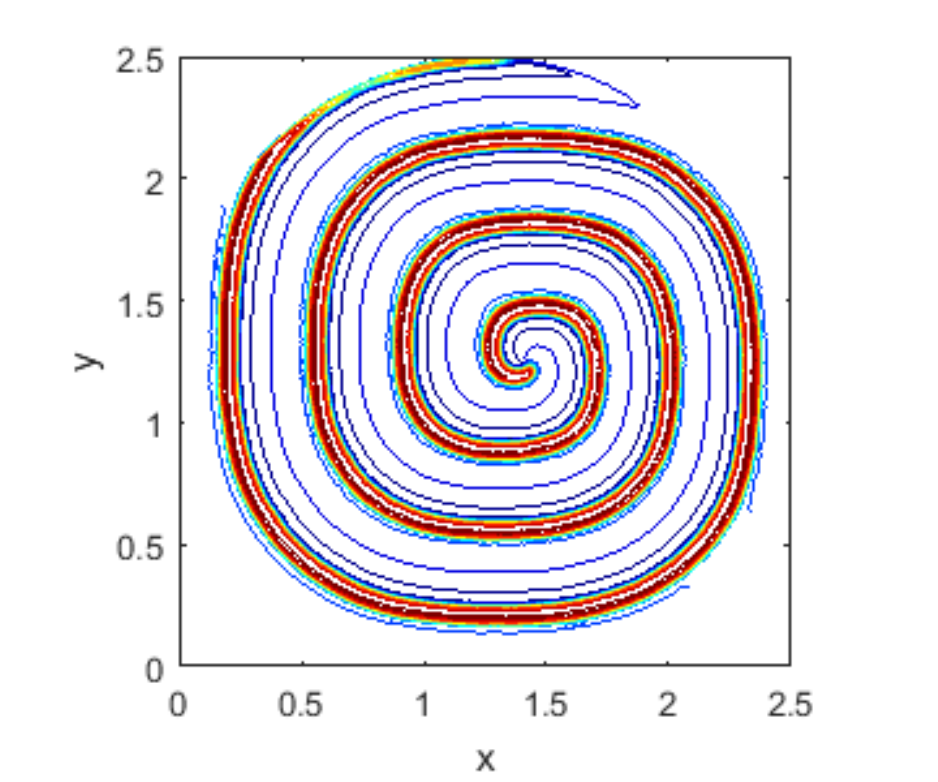} \\
      \includegraphics[width=0.9\textwidth]{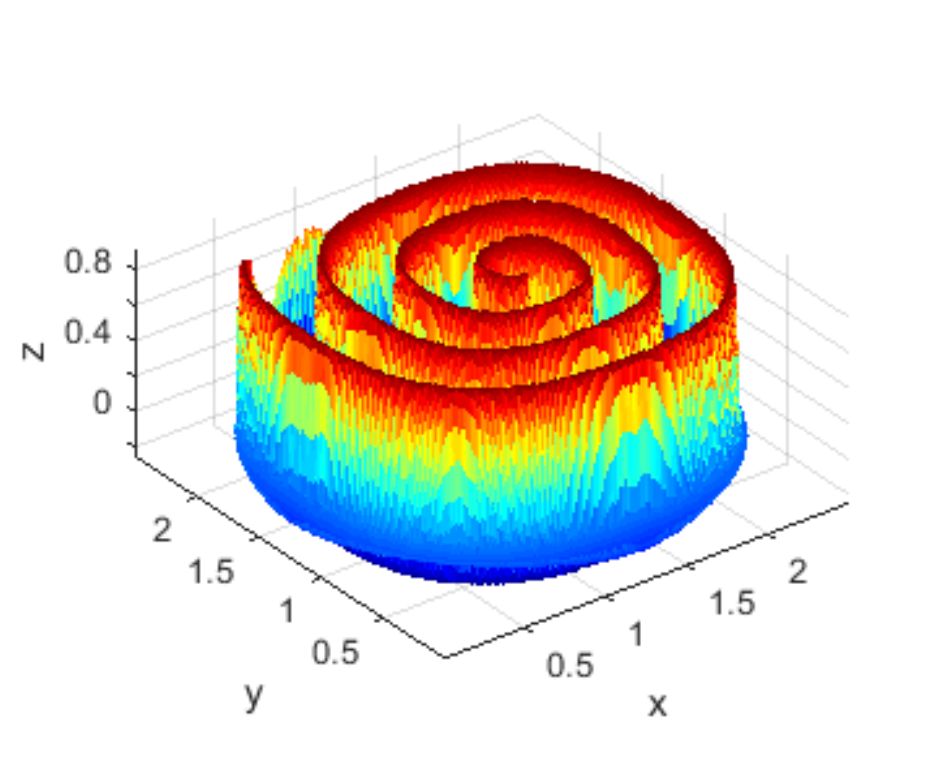}
    \end{minipage}}
  \subfloat[$\alpha=0.85,\beta=0.85$]{
    \begin{minipage}[t]{0.3\textwidth}
      \includegraphics[width=0.9\textwidth]{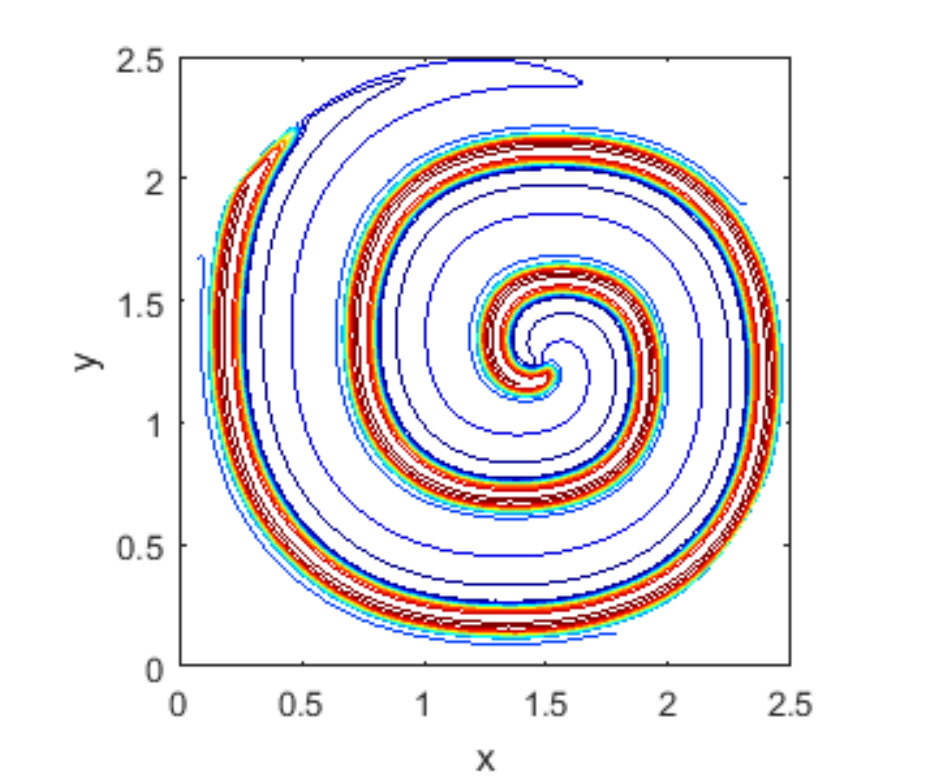} \\
      \includegraphics[width=0.9\textwidth]{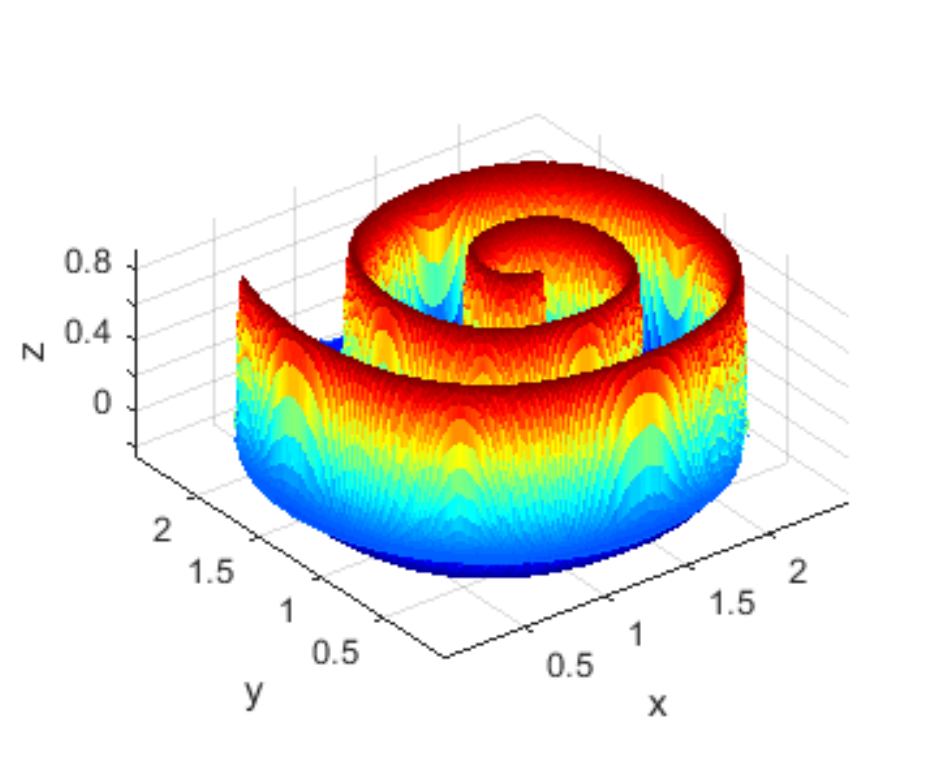}
    \end{minipage}}
  \subfloat[$\alpha=1,\beta=1$]{
    \begin{minipage}[t]{0.3\textwidth}
      \includegraphics[width=0.9\textwidth]{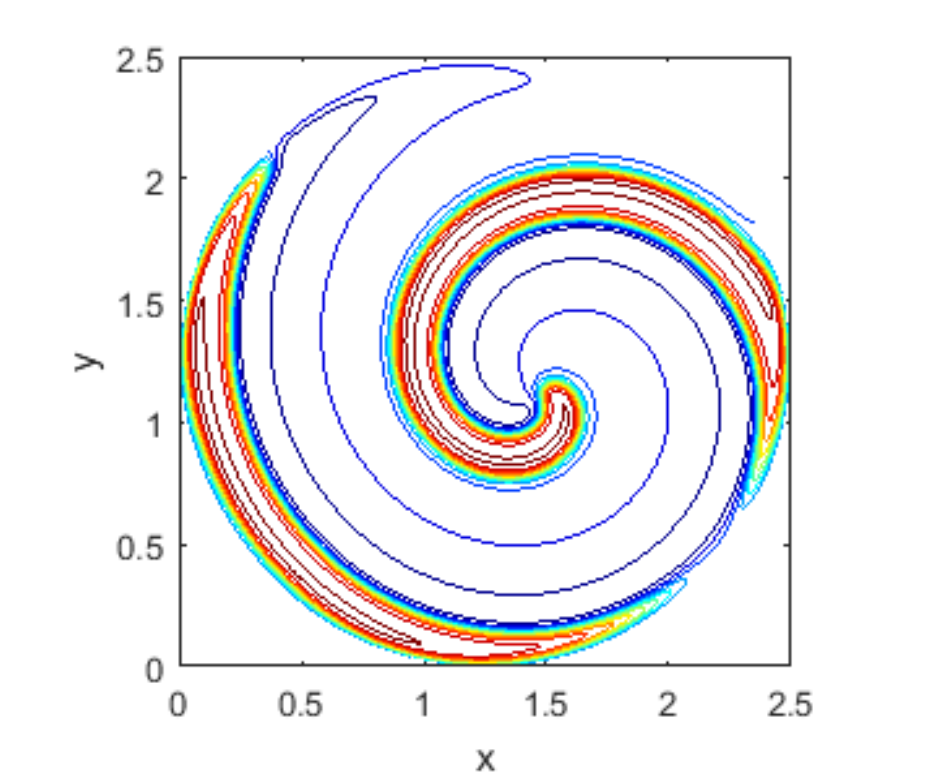} \\
      \includegraphics[width=0.9\textwidth]{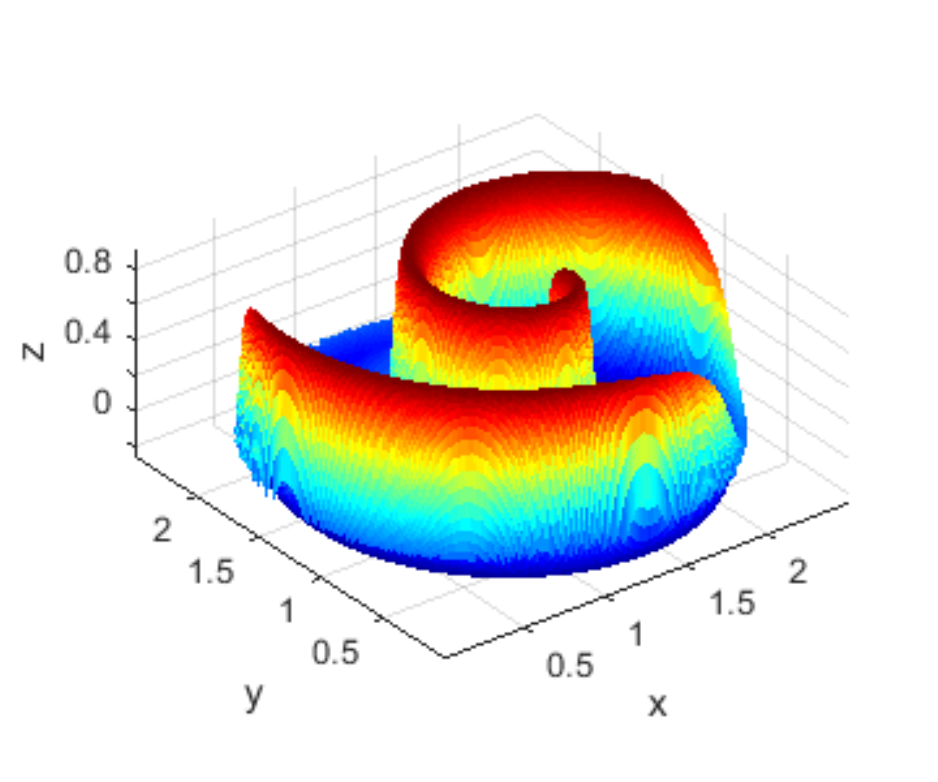}
    \end{minipage}}
  \caption{The simulation results of FitzHugh-Nagumo model when $T=1000$ 
    with $K_x=0.0001, K_y=0.0001$.}\label{fig:fhn_1}
\end{figure}
\begin{figure}
  \centering
  \subfloat[$\alpha=0.75,\beta=0.75$]{
    \begin{minipage}[t]{0.3\textwidth}
      \includegraphics[width=0.9\textwidth]{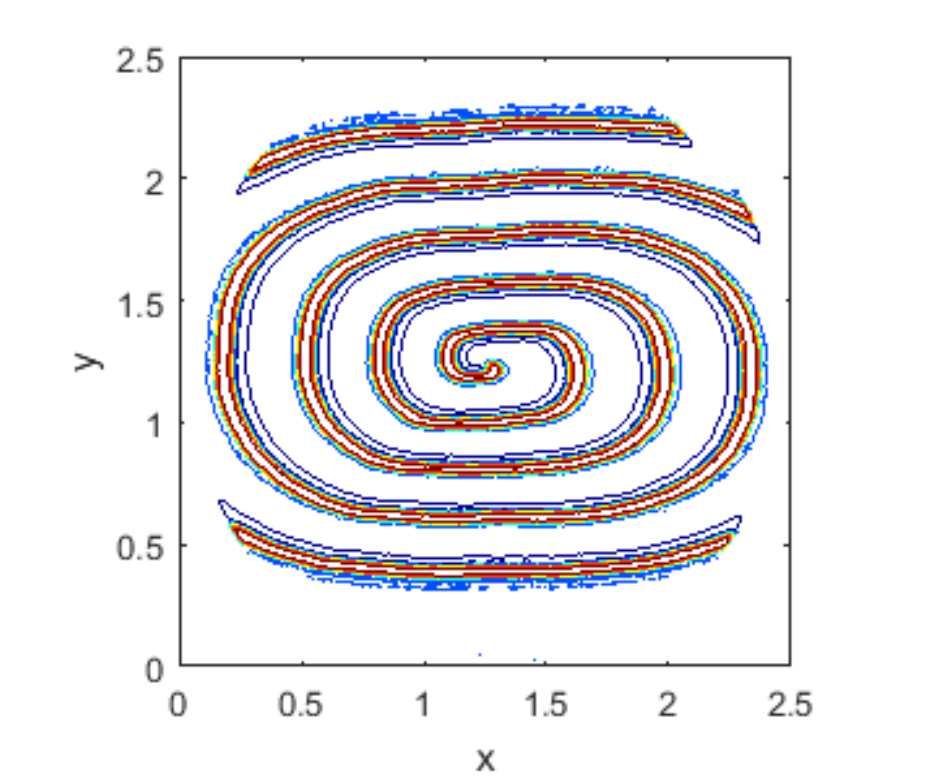} \\
      \includegraphics[width=0.9\textwidth]{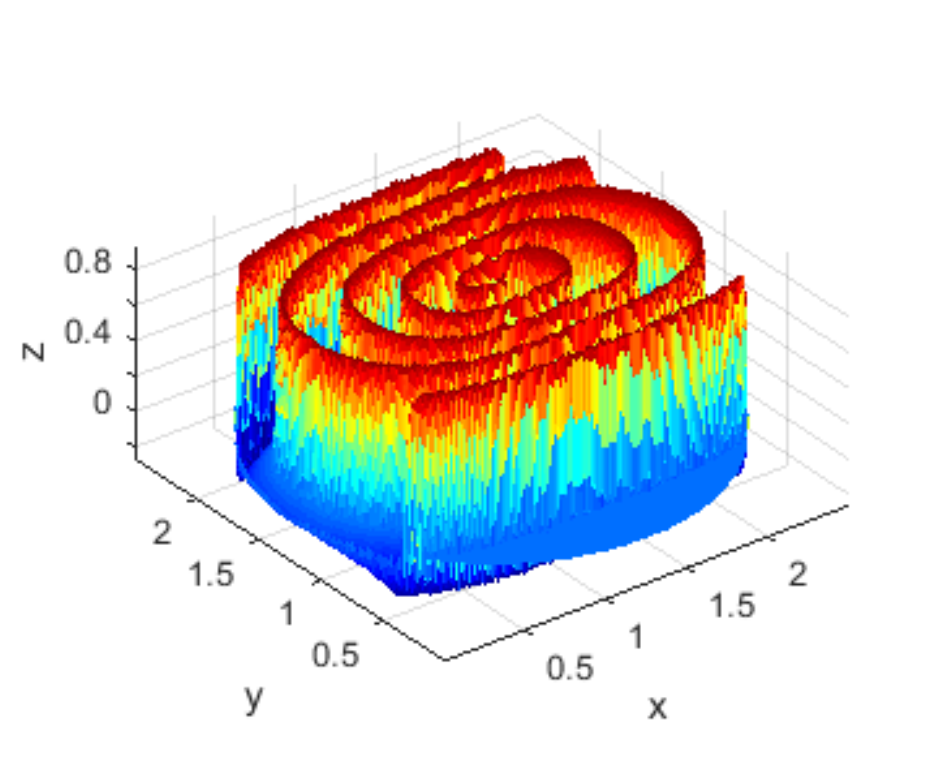}
    \end{minipage}}
  \subfloat[$\alpha=0.85,\beta=0.85$]{
    \begin{minipage}[t]{0.3\textwidth}
      \includegraphics[width=0.9\textwidth]{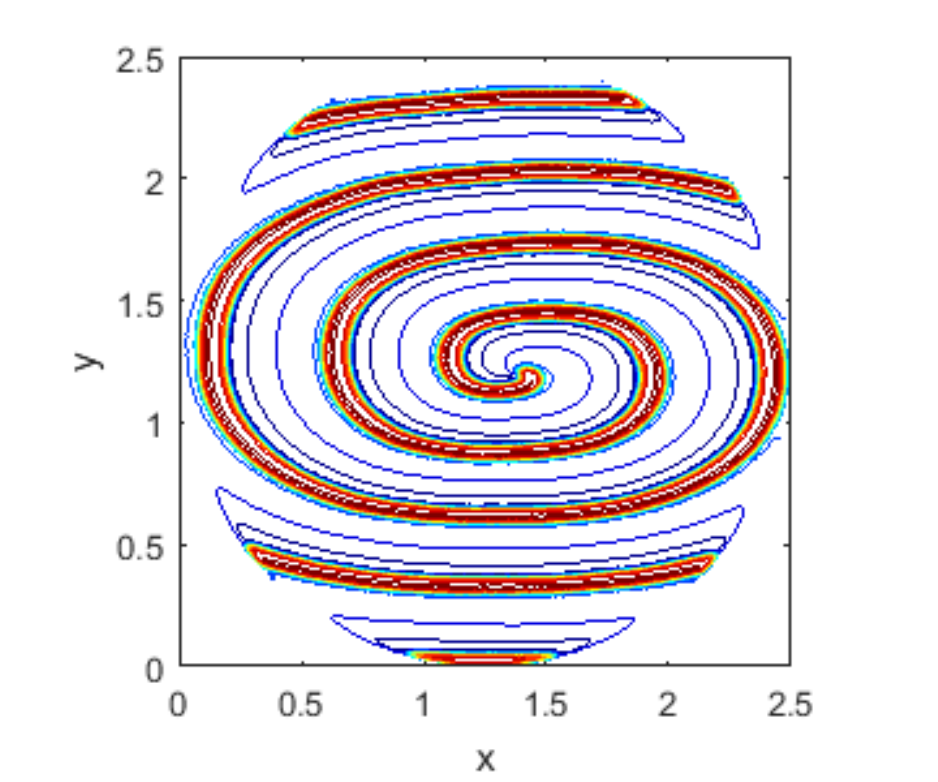} \\
      \includegraphics[width=0.9\textwidth]{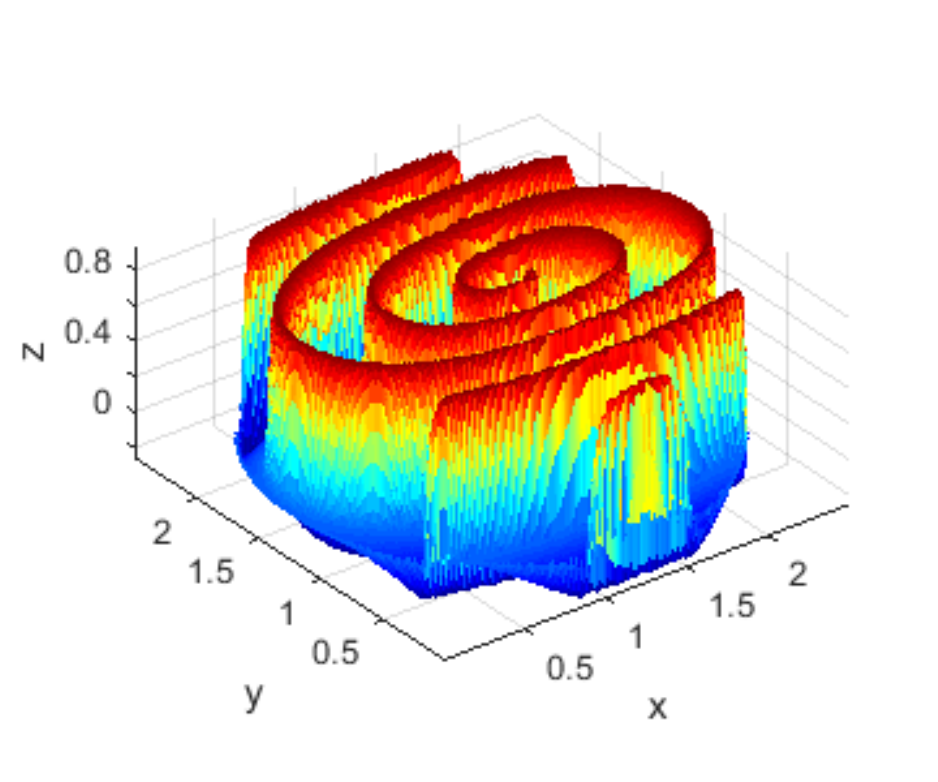}
    \end{minipage}}
  \subfloat[$\alpha=1,\beta=1$]{
    \begin{minipage}[t]{0.3\textwidth}
      \includegraphics[width=0.9\textwidth]{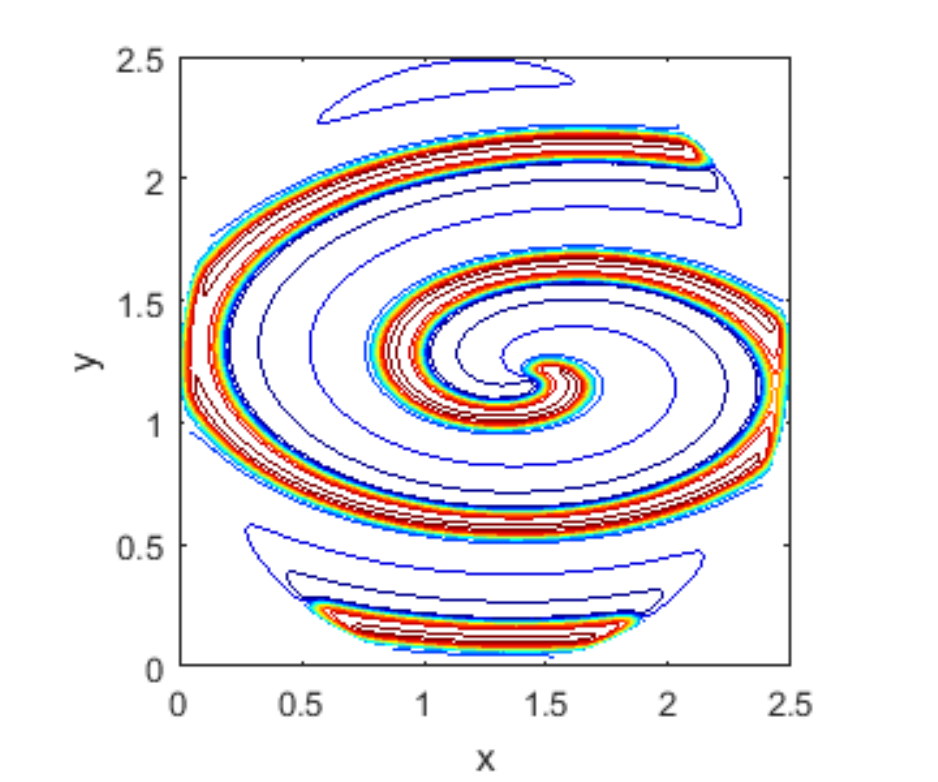} \\
      \includegraphics[width=0.9\textwidth]{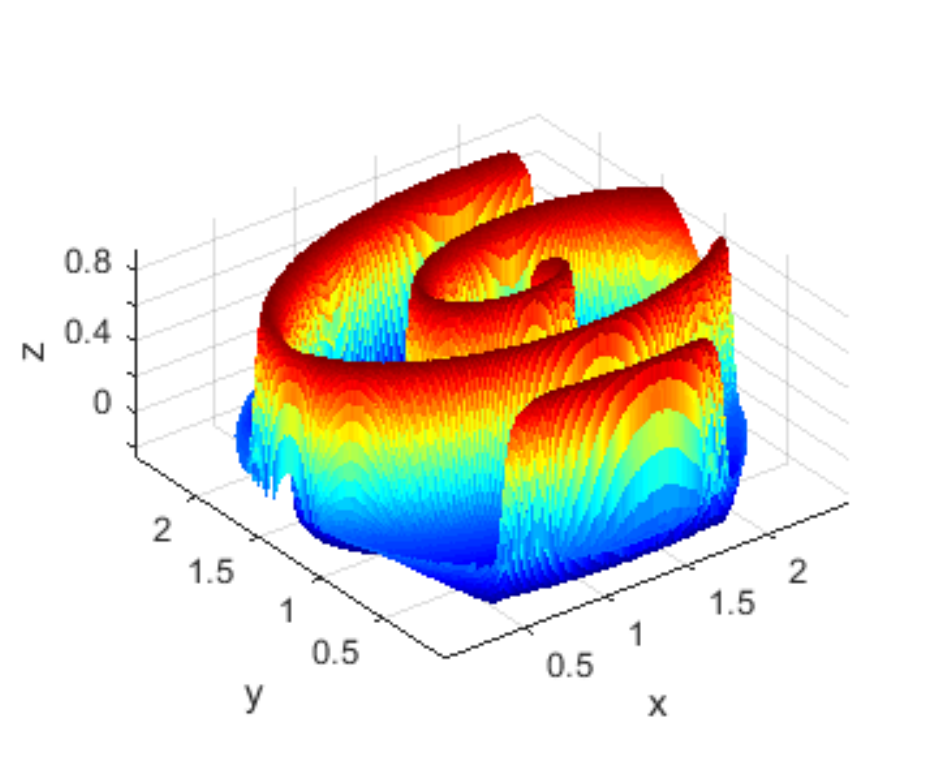}
    \end{minipage}}
  \caption{The simulation results of FitzHugh-Nagumo model when $T=1000$ 
    with $K_x=0.0001, K_y=2.5e-05$}\label{fig:fhn_2}
\end{figure}
\end{example}

\section{Conclusion}\label{sec:conclusion}
In this paper, we used Galerkin method to approach the nonlinear Riesz space
fractional diffusion equations on convex domain by approximate nonlinear term
with Taylor formula.  \mOne{This method has some advantages compared with the existing methods.}
It can be used to solve those problems on convex domain with
unstructured meshes, which is seldom solved before. 
Though it is introduced on convex domain, 
the implementation of our method can also be expanded to solve problems on non-convex domain.
And, from numerical tests, we find the linearization method is a very useful 
approach to approximate nonlinear term. \mOne{However, in the numerical tests, we have 
found computational cost of the Algorithm~\ref{algorithm1} increases nonlinearly as the
increase of elements. A simple way to speedup is using parallel algorithm  
because finding the integral paths of the Gaussian points in different
elements are independent. Other speedup methods to assembling fractional
stiffness matrix are still under investigation.}

Here, we just considered the homogeneous Dirichlet boundary conditions. 
In the following work, we will consider other boundary conditions, 
including non-homogeneous boundary conditions, Neumann boundary conditions. 
Furthermore, we will consider time-space fractional differential equations.
\section*{Acknowledgements}
This research was supported by the National Natural Science Foundation of China
(Grant No.11471262) and Project of
Scientific Research of Shaanxi (Grant No. 2015GY032).
\mMy{The authors thank the referees for their useful suggestions to improve this paper.}

\section*{Reference}
\bibliographystyle{elsarticle-num}
\bibliography{Reference.bib}
\end{document}